\input ./style/arxiv-general.cfg
\documentclass[aop,MSNbibl,seceqn,dvips]{arximspdf}
\makeatletter
   \@ifpackageloaded{graphicx}{}{\usepackage{graphicx}}
\makeatother

\doi{10.1214/15-AOP1041}
\volume{44}
\issue{4}
\pubyear{2016}
\firstpage{3032}
\lastpage{3075}
\docsubty{FLA}

\makeatletter
\def\sfrac#1#2{#1/#2}
\def\vfrac#1#2{(#1)/#2}
\def\afrac#1#2{#1/(#2)}

\newcommand{\rrvert}{\vert}
\newcommand{\rrVert}{\Vert}
\newcommand{\llvert}{\vert}
\newcommand{\llVert}{\Vert}
\renewcommand{\mid}{|}
\newcommand{\xrightarrow}[1]{\stackrel{#1}{-\!\!-\!\!\!\longrightarrow}}
\newcommand{\eqref}[1]{(\ref{#1})}
\newtheorem{teo}{Theorem}[section]
\newtheorem{cor}[teo]{Corollary}
\newtheorem{lem}[teo]{Lemma}
\newproclaim{example}[teo]{Example}
\newproclaim{rem}[teo]{Remark}
\newcommand{\Real}{\mathbb R}
\newcommand{\F}{\mathcal{F}}
\newcommand{\G}{\mathcal{G}}
\newcommand{\one}[1]{\mathbf{1}_{\{#1\}}}
\renewcommand{\P}{\mathbb{P}}
\newcommand{\Q}{\mathbb{Q}}
\newcommand{\E}{\mathbb{E}}
\newcommand{\sign}{\operatorname{sign}}
\makeatother

\begin{document}
\begin{frontmatter}

\title{Mixed Gaussian processes: A filtering approach}
\runtitle{Mixed Gaussian processes}

\begin{aug}
\author[A]{\fnms{Chunhao}~\snm{Cai}\ead[label=e1]{chunhao\_cai@nankai.edu.cn}},
\author[B]{\fnms{Pavel}~\snm{Chigansky}\thanksref{T2}\ead
[label=e2]{Pavel.Chigansky@mail.huji.ac.il}}
\and
\author[C]{\fnms{Marina}~\snm{Kleptsyna}\corref{}\thanksref
{T3}\ead[label=e3]{marina.kleptsyna@univ-lemans.fr}}
\runauthor{C. Cai, P. Chigansky and M. Kleptsyna}
\affiliation{Nankai University, The Hebrew University of Jerusalem and Universit\'{e} du Maine}
\address[A]{C. Cai\\
School of Mathematical Science\\
Nankai University\\
94 Weijin Road, Nankai\\
Tianjin 300071\\
China\\
\printead{e1}}
\address[B]{P. Chigansky\\
Department of Statistics\\
The Hebrew University of Jerusalem\\
Mount Scopus\\
Jerusalem 91905\\
Israel\\
\printead{e2}}
\address[C]{M. Kleptsyna\\
Laboratoire Manceau de Math\'{e}matiques\\
Universit\'{e} du Maine\\
Avenue Olivier Messiaen\\
72085 Le Mans\\
France\\
\printead{e3}}
\end{aug}
\thankstext{T2}{Supported by ISF 558/13 grant.}
\thankstext{T3}{Supported in part by ANR STOSYMAP.}

%
\received{\smonth{8} \syear{2014}}
%
\revised{\smonth{6} \syear{2015}}

%
\begin{abstract}
This paper presents a new approach to the analysis of mixed processes
\[
X_t = B_t + G_t, \qquad t\in[0,T],
\]
where $B_t$ is a Brownian motion and $G_t$ is an independent centered
Gaussian process.
We obtain a new canonical innovation representation of $X$, using
linear filtering theory.
When the kernel
\[
K(s,t) = \frac{\partial^2}{\partial s\,\partial t} \mathbb{E} G_t
G_s,\qquad s
\ne t %
\]
has a weak singularity on the diagonal, our results generalize the
classical innovation formulas beyond the square integrable setting.
For kernels with stronger singularity, our approach is applicable to
processes with additional ``fractional'' structure,
including the mixed fractional Brownian motion from mathematical finance.
We show how previously-known measure equivalence relations and
semimartingale properties follow from our canonical
representation in a unified way, and complement them with new formulas
for Radon--Nikodym densities.
\end{abstract}

\end{frontmatter}

\section{Introduction}\label{sec1}

In this paper, we present a new perspective on \textit{mixed} processes of
the form
%
\begin{equation}
\label{XBG} X_t = B_t + G_t, \qquad t
\in[0,T], T>0,
\end{equation}
where $B=(B_t)$ is a Brownian motion and $G=(G_t)$ is an independent
Gaussian process.
Such mixtures have been the subject of much research in the past, due
to their importance in
engineering applications (see, e.g., the survey \cite{KaiPoor98}),
and, more recently, have reemerged in
mathematical finance in the context of option pricing.

The renewed interest was triggered by Cheridito's paper \cite
{Ch01}, in which the author considered mixed fractional
Brownian motion (fBm)
%
\begin{equation}
\label{mfBm} X_t = B_t + B^H_t,
\qquad t\in[0,T],
\end{equation}
where $B^H=(B^H_t)$ is fBm with the Hurst exponent $H\in(0,1]$, that
is, the centered Gaussian process with
covariance function
%
\begin{equation}
\label{cov} R(s,t):=\E B^H_tB^H_s=
\tfrac{1} 2 \bigl( t^{2H}+ s^{2H}-\llvert t-s\rrvert
^{2H} \bigr), \qquad s,t\in[0,T].
\end{equation}
%

A curious change in properties of $X$ was shown to occur at $H=\frac{3} 4$,
where it became apparent that $X$ is a semimartingale
in its own filtration if, and only if, either $H=\frac{1} 2$ or $H\in
(\frac{3} 4, 1 ]$.
Moreover, in the latter case, the probability measure~$\mu^X$, induced
by $X$ on its paths space, is equivalent to the
Wiener measure $\mu^B$.

Since $B^H$ is not a semimartingale on its own, unless $H=\frac{1} 2$ or
$H=1$, this assertion means that $B^H$ can be ``regularized'' up to a
semimartingale
by adding to it an independent Brownian perturbation. In \cite{Ch01},
this fact is discussed in the context of arbitrage opportunities on
nonsemimartingale markets (see also \cite{Ch03}).
A~comprehensive survey of further related developments in finance can
be found in \cite{BSV07}.

Besides being of interest to the finance community, the result in \cite
{Ch01} also led to a number of elegant generalizations and alternative proofs
\cite{BN03,vZ07,vZ08}.
In addition, as pointed out in \cite{Ch03b}, the equivalence $\mu
^X\sim\mu^B$ follows from the general theory of Shepp \cite{Sh66},
Hitsuda \cite{Hi68} and Kailath \cite{Kai70}, which, moreover,
gives a formula for the density $d\mu^X/d\mu^B$.

A complementary result, obtained in \cite{BP88} (see Proposition 6.5)
and \cite{vZ07}, asserts that $\mu^X\sim\mu^{B^H}$ if, and only if,
$H<\frac{1} 4$. Both proofs are based on the spectral theory of
processes with stationary increments and the corresponding
density is given in \cite{vZ07} in terms of certain reproducing
kernels. However, as the author points out, a more explicit expression
might be hard to obtain using this method.

The main contribution of this paper is a novel approach to the analysis
of mixtures such as \eqref{XBG},
based on the filtering theory of Gaussian processes. The core of our
method is a new canonical innovation representation of $X$.
Our construction reveals a new interesting connection between the
probabilistic properties of $X$ and the structure
of solutions of integral equations with weakly singular kernels.

In the context of mixed fBm \eqref{mfBm}, all the aforementioned
properties can be deduced from this representation
in a unified manner for all values of $H$, due to an apposite choice of
the fundamental martingale. Moreover, it yields the missing density
${d\mu^X}/{d\mu^{B^H}}$ for $H<\frac{1} 4$,
as well as Girsanov-type formulas for the density of $\mu^X$ with
respect to measures induced by stochastic shifts of $X$.

The precise formulation of our results is given in the next section.
Section~\ref{sec-not} contains auxiliary results, including a relevant
theory of integral equations and frequently
used formulas from stochastic calculus with respect to fBm.
The proofs of the main theorems appear in Sections~\ref
{sec-thm1}--\ref{sec-thm4}, and in Section~\ref{sec-RL} we show how
our method applies to the
Riemann--Liouville process.

\section{The main results}

\subsection{A background}

Let us briefly recall the essential elements of the linear innovation
theory \cite{HH76,Roz77}.
A Gaussian process admits an \textit{innovation} representation if it can
be generated by a linear causal transformation of $N$
orthogonal processes with independent increments. Such a representation
is called \textit{canonical}, if the transformation is also causally invertible.
A well-known result of Hida \cite{Hida60} and Cramer \cite
{Cr64} asserts that under mild regularity conditions
any Gaussian process admits a canonical innovation representation.

Certain properties of canonical representations, collectively referred
to as \textit{type}, are the unique attributes of the process.
This includes the number $N$ of innovation components, called \textit
{multiplicity}, which can be finite or infinite.
For example, stationary processes have unit multiplicity, that is,
$N=1$, and the corresponding innovation representation can be found
by solving the spectral factorization problem.
Processes which induce equivalent measures on their paths space, have
the same innovation type (see \cite{KO73}).

Let us now review in greater detail the results 
directly relevant to the mixed processes of the form \eqref{XBG}.
A general criteria for equivalence, obtained by Shepp in \cite
{Sh66}, implies
that $\mu^X\sim\mu^B$ if and only if
%
\begin{equation}
\label{GK} \E G_t G_s = \int_0^s
\int_0^t K(u,v)\,du\,dv,
\end{equation}
with a kernel $K\in L^2 ([0,T]^2 )$. The corresponding formula
for the density $d\mu^X/d\mu^B$, given in \cite{Sh66},
involves the Carleman--Fredholm determinant and resolvent kernel of the
covariance operator, associated with $K$.

While Shepp's result gives a complete answer to the question of
equivalence, it
does not immediately reveal the innovation structure
of the process $X$. The missing link was found by Kailath in \cite{Kai70},
who noticed the relevance of factorization theory of Fredholm operators
in Hilbert spaces, developed by
Gohberg and Krein at around the same time.
Using the resolvent identity (7.10) from \cite{GK70}, Shepp's density
formula is rewritten in \cite{Kai70} in the form:
\[
\frac{d\mu^X}{d\mu^B}(X) = \exp\biggl(-\int_0^T
\varphi_t(X)\,dX_t-\frac{1} 2 \int
_0^T \varphi_t^2(X)\,dt
\biggr), %
\]
where $\varphi_t(X) = \int_0^t L(s,t) \,dX_s$ with $L\in
L^2([0,T]^2)$ being the unique solution of the Wiener--Hopf integral equation
%
\begin{equation}
\label{Leq} L(s,t) + \int_0^t L(r,t)
K(r,s) \,dr = -K(s,t), \qquad0\le s\le t\le T.
\end{equation}
It follows by Girsanov's theorem that the process
%
\begin{equation}
\label{BbarX} \overline{B}_t = X_t + \int
_0^t \int_0^s
L(r,s)\,dX_r \,ds
\end{equation}
is a Brownian motion. Moreover, it is shown in \cite{Kai70} that $\F
^X_t=\F^{\overline{B}}_t$ and
%
\begin{equation}
\label{XBbar} X_t = \overline{B}_t -\int
_0^t \int_0^s
\ell(r,s) \,d\overline{B}_r \,ds,
\end{equation}
where $\ell\in L^2([0,T]^2)$ solves the Volterra equation
%
\begin{equation}
\label{ellL} \ell(s,t) + \int_s^t \ell(r,t)
L(s,r)\,dr = L(s,t),\qquad0\le s\le t\le T.
\end{equation}
In particular, it follows that $X$ has unit multiplicity.

A different construction of canonical representation was given by
Hitsuda in \cite{Hi68}, where
$\mu^X\sim\mu^B$ is shown to hold if and only if $X$ can be
represented in the form \eqref{XBbar} with \textit{some} Brownian motion
$\overline{B}$ and
\textit{some} Volterra kernel $\ell\in L^2([0,T]^2)$. The representation
is proved to be canonical and the formula \eqref{BbarX} is obtained. Though
kernels $\ell$ and $L$ are characterized in \cite{Hi68} in a
different way, it can be shown that, in fact, both representations
coincide. A detailed
discussion about the links between all the aforementioned results can
be found in \cite{Ch03b}.

\subsection{A new canonical representation}
The canonical representation \eqref{BbarX} and \eqref{XBbar} requires
that the kernel $K$ belongs to $L^2([0,T]^2)$.
On the other hand, it is well known that $X$ may have multiplicity
greater than one if, for example, $K$ is only integrable on $[0,T]^2$. This
can be seen in a simple example.

\begin{example}\label{ex}
Consider the process
\[
X_t = B_t+\xi\int_0^t
\frac{1}{\sqrt{\llvert1-s\rrvert}}\,ds, %
\]
where $\xi\sim N(0,1)$ is independent of $B$.
It is easy to see that $\xi$
can be recovered precisely from $\F^X_t$ for all $t\ge1$.
Therefore, the filtration $\F^X_t$ is discontinuous at $t=1$, with $\F
^X_{t-} \varsubsetneq\F^X_t=\F^B_t\vee\sigma\{\xi\}$ for all
$t\ge1$.
By uniqueness of multiplicity, $X$ cannot be innovated by a single
Brownian motion on any interval $[0,T]$ with $T>1$. In this case, equation
\eqref{Leq} has no solution on $[0,T]$, even though $K\in L^1
([0,T]^2 )$.
In fact, discontinuity of filtration is not essential and it is
possible to construct $X$ with arbitrary multiplicity and
continuous natural filtration (Example~2 on page 266 in \cite{LS89}
and Example D on page 72 in \cite{HH76}).
\end{example}

The following theorem shows that $X$ has unit multiplicity under fairly
general conditions,
beyond the $L^2([0,T]^2)$ case, and gives the corresponding canonical
representation.

\begin{teo}\label{teo2}
Let $X$ be given by \eqref{XBG}, where $G$ satisfies \eqref{GK} with
%
\begin{equation}
\label{Kws} \bigl\llvert K(s,t)\bigr\rrvert\le C \bigl(1+\llvert
s-t\rrvert
^{-\alpha} \bigr), \qquad0\le\alpha<1,
\end{equation}
for some constant $C$.
Define
$
\phi_s = 1-\int_0^s L(r,s)\,dr$,
where $L(s,t)$ is the solution of equation \eqref{Leq}. Then the process
%
\begin{equation}
\label{spM} \overline{B}_t = \E\biggl(\int_0^t
\phi_s \,dB_s\Big| \F^X_t \biggr),
\end{equation}
is a Brownian motion, satisfying
%
\begin{equation}
\label{barBq}\overline{B}_t = \int_0^t
q(s,t)\,dX_s,
\end{equation}
with $q(s,t)$ being the unique solution of the Wiener--Hopf equation:
%
\begin{equation}
\label{Qeq} q(s,t) +\int_0^t q(r,t)
K(r,s) \,dr = \phi(s), \qquad0\le s,t\le T.
\end{equation}
The representation
%
\begin{equation}
\label{XqB} X_t = \int_0^t
\hat q(s,t)\,d\overline{B}_s
\end{equation}
with
$
\hat q(s,t) = - \frac{\partial}{\partial s} \int_s^t q(r,s)\,dr$,
is canonical, that is, $\F^X_t=\F^{\overline{B}}_t$.
\end{teo}

\begin{rem}
1.~As\vspace*{1pt} can be seen from the proof in Section~\ref{sec-thm2}, the
assertion of this theorem remains true 
when $K\in L^2([0,T]^2)$, without assuming the particular structure of
\eqref{Kws}.
Moreover,
$
\frac{\partial}{\partial t}q(s,t) = L(s,t)
$
and, therefore,
\[
\overline{B}_t = \int_0^t
q(s,t)\,dX_s = X_t + \int_0^t
\int_s^t L(s,r)\,dr \,dX_s. %
\]
When $K$ is square integrable, so is the solution $L$ of \eqref{Leq}
and the order of integration in the right-hand side can be
interchanged, recovering
the formula \eqref{BbarX}. Moreover, in this case $\hat
q(t,t)=q(t,t)=1$ and
\[
\hat q(s,t) = 1-\int_s^t L(r,s)\,dr %
\]
and hence
%
\begin{eqnarray}
\label{Xfla} X_t &=&\int_0^t
\hat q(s,t)\,d\overline{B}_s
=
\overline{B}_t-\int_0^t \int
_s^t L(r,s)\,dr\,d\overline{B}_s
\nonumber\\[-8pt]\\[-8pt]\nonumber
&=&
\overline{B}_t-\int_0^t \int
_0^r L(r,s) \,d\overline{B}_s\,dr.
\end{eqnarray}

Comparing this with \eqref{XBbar} reveals a curious relation between
the solutions of the Volterra equation
\eqref{ellL} and the Wiener--Hopf equation \eqref{Leq}: the solution
of the former on the sub-diagonal coincides with the solution of the
latter on the \textit{super-diagonal}, that is, $\ell(s,t)=L(t,s)$ for
$s\le t$.
In\vspace*{1pt} other words, both the direct and the inverse transformation between
$X$ and $\overline{B}$ can be expressed in terms of the single Wiener--Hopf
equation \eqref{Leq}, whose solution is extended to the whole
rectangle $[0,T]^2$.

2.~A two stage procedure can be used to construct a canonical
representation for the process
\[
X_t = B_t + G_t + G^\dagger_t,
\]
where $G^\dagger$ is an independent centered Gaussian process,
satisfying \eqref{GK} with $K^\dagger\in L^2([0,T]^2)$.
We can first generate an intermediate process $\overline{X}$ by
applying \eqref{barBq}:
\[
\overline{X}_t = \int_0^t
q(s,t)\,dX_s = \overline{B}_t + \int_0^t
q(s,t)\,dG^\dagger_s. %
\]
It can readily be seen that the process, defined by the last term,
satisfies \eqref{GK} with a square integrable kernel and, therefore,
can be represented canonically in the standard way.

3.~When $K\in L^2([0,T]^2)$, it follows from the results of Shepp
and Hitsuda,
that the kernel $\ell(s,t)$ also solves the Riccati--Volterra equation
%
\begin{equation}
\label{RV} \ell(s,t)=K(s,t)-\int_0^{t\wedge s}
\ell(s,r)\ell(t,r)\,dr.
\end{equation}
If $ G_t = \int_0^t \gamma_s\,ds$ with a Gaussian process
$\gamma$, then the innovating Brownian motion
reduces to
%
\begin{equation}
\label{Kailath} \overline{B}_t = X_t -\int
_0^t \pi_s(\gamma)\,ds,
\end{equation}
with $\pi_t(\gamma)=\E(\gamma_t\mid\F^X_t)=\int_0^t \ell
(s,t)\,d\overline{B}_s$.
In this form, due to  Kailath \cite{Kai68}, the canonical
representation plays an important role in control and filtering theory,
which goes beyond the linear setting (see \cite{He11}).
If furthermore, $\gamma$ is the Gauss--Markov process, \eqref{RV}
reduces to the familiar Riccati equation from the Kalman--Bucy filter.

Note that in our approach conditioning on $\F^X_t$ is used in an
essentially different way than in \eqref{Kailath}:
in the lack of derivative of $G_t$, the innovation Brownian motion is
produced by projecting a specially designed martingale \eqref{spM}.
Somewhat unexpectedly, equality of filtrations can be established in
this case, using only the basic theory of integral equations,
which needs nothing more than weak singularity of~$K$.

4.~Condition \eqref{Kws} is borrowed from the classical theory
of integral equations (see Section~\ref{sec-wseq} below).
It is satisfied by some interesting processes, related to the fBm. One
example is {\em bifractional} Brownian motion, introduced in \cite
{HV02}, which
is a centered Gaussian process $G$ with covariance function
\[
\E G_t G_s = \frac{1} {2^K} \bigl(
\bigl(t^{2H}+s^{2H}\bigr)^K - \llvert t-s\rrvert
^{2HK} \bigr), %
\]
where $H\in(0,1)$ and $K\in(0,1]$. The representation \eqref{GK}
holds for $HK >\frac{1} 2$ with
%
\begin{eqnarray}
\label{KbiBm}
K(s,t) &=& C_1(H,K) \bigl(t^{2H}+s^{2H}
\bigr)^{K-2}s^{2H-1} t^{2H-1}
\nonumber\\[-8pt]\\[-8pt]\nonumber
&&{} +
C_2(H,K)\llvert t-s\rrvert^{2HK-2}, \qquad t\ne s,
\end{eqnarray}
where
$
C_1(H,K)=\frac{K(K-1)(2H)^2} {2^K}<0
$
and
$
C_2(H,K)=\frac{2HK(2HK-1)} {2^K}>0$.
Since $HK>\frac{1} 2$ implies $H>\frac{1} 2$ and $2HK-2\in(-1,0)$,
\begin{eqnarray*}
\bigl(t^{2H}+s^{2H} \bigr)^{K-2}s^{2H-1}
t^{2H-1} &\le& (s\wedge t )^{2H-1} (s\vee t )^{2H(K-1)-1}
\\
&=&
\biggl( \frac{s\wedge t}{s\vee t} \biggr)^{2H-1} (s\vee t )^{2HK-2} \le
\llvert s- t\rrvert^{2HK-2}
\end{eqnarray*}
and thus the kernel \eqref{KbiBm} satisfies \eqref{Kws} with $\alpha
:= 2HK-2$.

Other examples are the \textit{sub-fractional} Brownian motion from
\cite
{BGT04} and the Riemann--Liouville process, which can be fitted in by a
similar calculation.
\end{rem}

Let us now return to the mixed fBm \eqref{mfBm}. For $H>\frac{1} 2$,
$B^H$ satisfies \eqref{GK} with
%
\begin{equation}
\label{K} K_H(s,t) = \frac{\partial^2}{\partial s\,\partial t} \E B^H_tB^H_s
= H(2H-1) \llvert s-t\rrvert^{2H-2}.
\end{equation}
For $H>\frac{3} 4$, this kernel is square integrable and, therefore, the
mixed fBm $X=B+B^H$ can be represented canonically by
\eqref{BbarX}--\eqref{XBbar}. For $H\in(\frac{1} 2, \frac{3} 4]$,
$K_H$ does not belong to $L^2([0,T]^2)$, but
still satisfies the assumptions of Theorem~\ref{teo2}. Therefore, its
canonical representation is given by \eqref{barBq} and \eqref{XqB}
for all values of $H$ in $(\frac{1} 2,1]$.

For $H<\frac{1} 2$, the kernel $K_H$ has stronger singularity than
admitted by \eqref{Kws} and
the covariance of $B^H$ fails to satisfy \eqref{GK}. Nevertheless,
quite remarkably (see Example~\ref{exinno}), the mixed fBm
can still be innovated canonically by a different martingale [cf.
\eqref{spM}].

\begin{teo}\label{teo1}
\textup{(i)} Let $X$ be defined by \eqref{mfBm}. The martingale $M_t = \E
(B_t\mid\F^X_t )$ admits the representation
%
\begin{equation}
\label{MX} M_t = \int_0^t
g(s,t) \,dX_s,\qquad\langle M\rangle_t=\int
_0^t g(s,t)\,ds,
\end{equation}
where $g(s,t)$ is the unique solution of the integro-differential equation
%
\begin{equation}
\label{ideq} g(s,t) + \frac{\partial}{\partial s}\int_0^t
g(r,t) \frac{\partial}{\partial r}R(r,s) \,dr=1, \qquad0<s,t\le T,
\end{equation}
with $R(s,t)$ defined in \eqref{cov}.

\textup{(ii)}~The quadratic variation of $M$ is given by
%
\begin{equation}
\label{Mqv} \qquad\frac{d}{dt} \langle M \rangle_t =
g^2(t,t) + \frac{2-2H}{\lambda
_H} \bigl(t^{1/2-H}(\Psi g) (t,t)
\bigr)^2 > 0, \qquad t\in[0,T],
\end{equation}
where $\lambda_H$ is the constant defined in \eqref{constants} and
$\Psi$ is the operator, defined in \eqref{Psi} below. The innovation
representation
%
\begin{equation}
\label{XG} X_t = \int_0^t
\hat g(s,t)\,dM_s,
\end{equation}
with
%
\begin{equation}
\label{tildeg} \hat g(s,t) = 1- \frac{d}{d\langle M \rangle_s}\int_0^t
g(r,s)\,dr,
\end{equation}
is canonical, that is, $\F^X_t=\F^M_t$ for all $t\in[0,T]$.
\end{teo}

\begin{rem}
1.~For $H>\frac{1} 2$ the kernel $K_H$ in \eqref{K} is integrable and,
therefore, the derivative in \eqref{ideq} can be interchanged with integration,
so that it takes the form of the Wiener--Hopf integral equation [cf.
\eqref{Qeq}]
%
\begin{equation}
\label{WHeq} g(s,t) + \int_0^t g(r,t)
K_H(r,s) \,dr=1, \qquad0\le s,t\le T.
\end{equation}
In this case, the second term in \eqref{Mqv} vanishes [see \eqref
{Psi} below]
and \eqref{tildeg} becomes
%
\begin{equation}
\label{tildegg}\qquad \hat g(s,t) = 1- \frac{1} {g(s,s)} \int_0^t
\frac{ {\partial
}/{\partial s} g(r,s)}{g(s,s)}\,dr=1- \frac{1} {g(s,s)} \int_0^t
L(r,s) \,dr.
\end{equation}
The last equality holds, since $\frac{ {\partial}/{\partial s}
g(r,s)}{g(s,s)}$ turns to be the solution of equation \eqref{Leq}.
Differentiating \eqref{tildegg} yields
$
\frac{\partial}{\partial t}\hat g(s,t) = -\frac{L(t,s) }{g(s,s)}$,
and we obtain
\begin{eqnarray*}
X_t &= & \int_0^t \hat
g(s,t)\,dM_s = \int_0^t \hat
g(s,s)\,dM_s + \int_0^t \int
_s^t\frac{\partial}{\partial\tau}\hat g(s,\tau)\,d\tau
\,dM_s
\\
&=&
\int_0^t \hat g(s,s)\,dM_s -
\int_0^t \int_s^t
\frac{L(\tau
,s)}{g(s,s)} \,d\tau \,dM_s.
\end{eqnarray*}
A calculation shows that $\hat g(s,s)=1/g(s,s)$ and, therefore,
\[
X_t = \overline{B}_t - \int_0^t
\int_s^t L(\tau,s) \,d\tau \,d\overline{B}_s,
\]
where
$
\overline{B}_t = \int_0^t \frac{1}{g(s,s)}\,dM_s$,
is a Brownian motion. Similar calculations give [cf.~\eqref{Xfla}]
\[
\overline{B}_t = X_t + \int_0^t
\int_s^t L(s,r) \,dr \,dX_s,
\]
and hence we are back to the innovation representation from Theorem
\ref{teo2}.

2.~Natural, as it may seem, the choice of the martingale $M_t = \E
(B_t\mid\F^X_t )$ is not at all obvious and,
in fact, it fails to innovate $X$ in general, as demonstrated in the
following example.

\begin{example}\label{exinno}
Let $m(t)=6^{1/3}\wedge t$ and $\xi(t)=\eta m(t)$, where the random
variable $\eta\sim N(0,1)$ is independent of $B$.
The martingale $M_t=\E(B_t\mid\F^X_t)$ still satisfies \eqref{MX} where
$g(s,t)$ solves the Wiener--Hopf equation \eqref{WHeq},
with $K_H(s,t)$ replaced by $K(s,t)= m(s)m(t)$. Its quadratic variation is
$
\langle M\rangle_t = \int_0^t g^2(s,s)\,ds$,
as in Theorem~\ref{teo1} for $H>\frac{1} 2$.

For the degenerate kernel $K(s,t)=m(s)m(t)$, equation \eqref{WHeq} can
be solved explicitly:
\[
g(s,t) = 1-m(s)\frac{\int_0^t m(r)\,dr}{1+\int_0^t m^2(r)\,dr}, \qquad0\le
s\le t\le T, %
\]
and an easy calculation shows that $g(t,t)=0$ for all $t\ge6^{1/3}$.
On the other hand, since $K\in L^2([0,T]^2)$ the representation \eqref
{BbarX}--\eqref{XBbar} is canonical and
therefore $X$ cannot be innovated by $M$, that is, $\F^M_t
\varsubsetneq\F^X_t$ for $t\ge6^{1/3}$.
Incidentally, $\{M_t, t\in[0,T]\}$ is a sufficient statistic in the
problem of estimating $\theta\in\Real$ from
the observations of $\{\theta t + X_t, t\in[0,T]\}$.
\end{example}

The equality $\F^X_t=\F^M_t$ in Theorem~\ref{teo1} is closely
related to Krein's method of solving integral equations
with difference kernels (see Theorem~\ref{Kthm} below). Krein showed
that the solution of the Wiener--Hopf equation with a unit forcing
function such as
\eqref{WHeq} does not vanish on the diagonal and that it can be used
to express solutions to this equation with an arbitrary right-hand side.
Remarkably, this property remains true for kernels with a somewhat more
general structure (Lemma~\ref{cor}), arising in the case of mixed fBm.
Thus, Theorem~\ref{teo1} gives a probabilistic interpretation of the
nondegeneracy of Krein's solutions in terms of the equality of filtrations.

3.~For $H<\frac{1} 2$, the kernel $K_H$ in \eqref{K} has a stronger
singularity and, consequently, the derivative and integration in
equation \eqref{ideq} are no longer interchangeable and the integral
equation \eqref{WHeq}
makes no sense. Nevertheless, \eqref{ideq} can still be solved by
reduction to a different weakly singular integral equation, using tools
from fractional calculus. Moreover, it turns out that, while the first
term in \eqref{Mqv} vanishes in this case, the second term remains
strictly positive for all $t\in[0,T]$ and, consequently, the
martingale $M$ generates the same filtration as $X$. Martingales with such
property are sometimes referred to as {\em fundamental} in fBm
literature (see, e.g., \cite{NVV99}), playing the central role in
related statistical problems.

Construction of a canonical representation for more general mixed
Gaussian processes of the form \eqref{XBG}, beyond the condition
\eqref{Kws},
seems to be quite a delicate problem, especially in view of Example
\ref{ex}.
Our approach remains applicable to other processes of ``fractional'' type.
One example is the Riemann--Liouville process (see \cite{MR99}):
%
\begin{equation}
\label{RLproc} V^H_t = 2H \int_0^t
(t-s)^{H-1/2}\,dV_s,
\end{equation}
where $V$ is a Brownian motion.
While many properties of $V^H$ are similar to those of $B^H$, there are
some essential differences, at least from the standpoint of the problems
under consideration.

First, the increments of $V^H$ are not stationary, and hence the
equivalence of $X=B + V^H$ and $V^H$ for $H<\frac{1} 4$,
cannot be deduced by the spectral technique, used in \cite{BP88} and
\cite{vZ07}. Second, for $H<\frac{1}2$
the first partial derivative
$
\partial/\partial s \E V^H_tV^H_s
$
already has a nonintegrable singularity of the diagonal and
consequently, equation~\eqref{ideq} makes sense only if the inner
derivative is moved to
the solution itself. Further details are referred to Section~\ref{sec-RL}.
\end{rem}

\subsection{Equivalence relations and density formulas}

\subsubsection{The mixed fBm}
As discussed in the \hyperref[sec1]{Introduction}, our interest in
mixed fBm was motivated by the equivalence
relations, discovered in \cite{Ch01} and \cite{vZ07}. We will show
how these results can be derived from the
canonical representation of Theorem~\ref{teo1} and, in addition,
complement them with a new formula for the Radon--Nikodym density
in the case $H<\frac{1} 4$.

\begin{teo}\label{teo3}
\textup{(i)}~The process $X$ defined in \eqref{mfBm} is a semimartingale in its own
filtration if
and only if $H\in\{\frac{1} 2\} \cup(\frac{3} 4,1]$.
For $H\in(\frac{3} 4,1]$, $X$ is a diffusion type process
\[
X_t=\overline{B}_t -\int_0^t
\varphi_s(X)\,ds,\qquad t\in[0,T], %
\]
where $\overline{B}$ is a Brownian motion with $\F^{\overline B}_t=\F^X_t$,
$\varphi_t(X) = \int_0^t L(s,t) \,dX_s$ and
\[
L(s,t) := \frac{\partial}{\partial t} g(s,t) \Big/\sqrt{\frac{d}{dt}
\langle M
\rangle_t}. %
\]
The measures $\mu^X$ and $\mu^B$ are equivalent, if and only if $H\in
(\frac{3} 4, 1 ]$, and
\[
\frac{d\mu^X}{d\mu^B}(X) = \exp\biggl\{-\int_0^T
\varphi_t(X)\,dX_t-\frac{1} 2\int
_0^T \varphi_t^2(X)\,dt
\biggr\}. %
\]

\textup{(ii)}
For $H\in(0,\frac{1} 4 )$, $X$ is a fractional diffusion type process
%
\begin{equation}
\label{BHt} X_t =\overline{B}^H_t-\int
_0^t \rho(s,t) \varphi_s(X)\,ds,
\end{equation}
where $\overline{B}^H$ is fBm with $\F^{\overline{B}^H}_t=\F^X_t$,
$\varphi_t(X) = \int_0^t L(s,t) \,dX_s$ and
\[
L(s,t) := \frac{\partial}{\partial t} g(s,t) \Big/\sqrt{\frac{d}{dt}
\langle M
\rangle_t}-\frac{\partial}{\partial t}\tilde{\rho}(s,t), %
\]
with the kernels $\rho(s,t)$ and $\tilde\rho(s,t)$ are defined
in \eqref{rhorho} below.
The measures $\mu^X$ and $\mu^{B^H}$ are equivalent if and only if
$H\in(0,\frac{1} 4 )$ and
%
\begin{equation}
\label{dmuBh} \frac{d\mu^X}{d\mu^{B^H}}(X) = \exp\biggl\{-\int_0^T
\varphi_t( X)\,d \widetilde{X}_t -\frac{1} 2\int
_0^T \varphi_t^2( X)\,dt
\biggr\},
\end{equation}
where $ {\widetilde X_t = \int_0^t \tilde{\rho}(s,t)\,dX_s}$.
\end{teo}

\begin{rem}
1.~The density formulas in both cases are given in terms of
solutions of certain integral equations, rather than
reproducing kernels as in \cite{vZ07}.
Work in progress indicates that in some statistical applications, such
as estimating
$H>\frac{3}4$ from the sample $X^T = \{X_t, t\in[0,T]\}$, integral
equations are a more manageable alternative.

2.~A similar result holds for the mixed Riemann--Liouville
process $X = B + V^H$ with $V^H$ defined
in \eqref{RLproc}. The precise details appear in Section~\ref{sec-RL} below.
\end{rem}

\subsubsection{Mixed processes with drift}
The canonical representation also yields Girsanov's-type formulas,
useful in the likelihood based statistical inference.
Consider the process
%
\begin{equation}
\label{YfX} Y_t = \int_0^t
\xi_s\,ds + X_t, \qquad t\in[0,T],
\end{equation}
where $X$ is defined in \eqref{mfBm} and $\xi= (\xi_t)$ is a process with
continuous paths, satisfying $\E\int_0^T \llvert\xi_t\rrvert \,dt
<\infty
$. Assume that $\xi$ is adapted to a filtration $\G=(\G_t)$,
with respect to which $M$, introduced in Theorem~\ref{teo1}, is a martingale.

The choice of the filtration $\G$ can vary in different applications.
For example, in filtering problems $\xi$ plays the role of
unobserved state process and $X$ is interpreted as the observation
noise. If the state process and the noise are independent, then the assumption
holds with $\G_t := \F^\xi_t\vee\F^X_t$.

If $\xi_t$ is a function of $Y_t$, then \eqref{YfX} becomes a
stochastic differential
equation with respect to the mixed fBm $X$. In this case, $\xi$ is
adapted to $\F^X_t$ itself, and hence the natural choice is $\G_t
:=\F^X_t$.
For example, $\xi_t: = a Y_t$ with $a \in\Real$ corresponds to the
mixed fractional
Ornstein--Uhlenbeck process
\[
Y_t = a \int_0^t Y_s
\,ds + X_t, \qquad t\in[0,T], %
\]
with the drift parameter $a$.

Theorem~\ref{teo1} yields a formula for the density of $\mu^Y$ with
respect to $\mu^X$:

\begin{cor}\label{teo4}
The process $Y$ admits the representation
%
\begin{equation}
\label{repYG} Y_t = \int_0^t
\hat g(s,t)\,dZ_s
\end{equation}
with $\hat g(s,t)$ defined in \eqref{tildeg},
where
\[
Z_t = \int_0^t
g(s,t)\,dY_s,\qquad t\in[0,T] %
\]
is a $\G$-semimartingale with decomposition
%
\begin{equation}
\label{reprZ} Z_t = M_t + \int_0^t
\Xi(s)\,d\langle M\rangle_s
\end{equation}
and
%
\begin{equation}
\label{Kr2} \Xi(t) = \frac{d}{d\langle M \rangle_t}\int_0^t
g(s,t)\xi_s \,ds.
\end{equation}
In particular, $\F^Y_t=\F^Z_t$, for all $t\in[0,T]$ and, if
\[
\E\exp\biggl\{-\int_0^T
\Xi(t)\,dM_t-\frac{1} 2\int_0^T
\Xi^2(t) \,d\langle M\rangle_t \biggr\}=1, %
\]
then $\mu^X\sim\mu^Y$ and
%
\begin{equation}
\label{RNflaXY} \frac{d\mu^Y}{d\mu^X}(Y) = \exp\biggl\{\int_0^T
\widehat{\Xi}(t)\,dZ_t-\frac{1} 2\int_0^T
\widehat{\Xi}^2(t) \,d\langle M\rangle_t \biggr\},
\end{equation}
where $\widehat{\Xi}(t)=\E(\Xi(t)\mid\F^Y_t )$.
\end{cor}

In the setting of Theorem~\ref{teo2}, we have the following analog.

\begin{cor}\label{teo4b}
Let $Y$ be the process in \eqref{YfX}, where $X$ is defined in \eqref
{XBG} and satisfies the assumptions of Theorem~\ref{teo2}.
Then $Y$ admits the representation
\[
Y_t = \int_0^t \hat q(s,t)
\,dZ_s, 
\]
where the process
$
Z_t = \int_0^t q(s,t) \,dY_s 
$
satisfies
\[
Z_t = \overline{B}_t + \int_0^t
\Xi(s)\,ds, %
\]
with $ \Xi(s) = \xi_s + \int_0^s L(u,s)\xi_u\,du$.
In particular, $\F^Y_t=\F^Z_t$, for all $t\in[0,T]$ and, if
\[
\E\exp\biggl\{-\int_0^T \Xi(t)\,d
\overline{B}_t-\frac{1} 2\int_0^T
\Xi^2(t) \,dt \biggr\}=1, %
\]
then $\mu^X\sim\mu^Y$ and
\[
\frac{d\mu^Y}{d\mu^X}(Y) = \exp\biggl\{\int_0^T
\widehat{\Xi}(t)\,dZ_t-\frac{1} 2\int_0^T
\widehat{\Xi}^2(t) \,d t \biggr\}, %
\]
where $\widehat{\Xi}(t)=\E(\Xi(t)\mid\F^Y_t )$.
\end{cor}

\section{Notation and auxiliary results}\label{sec-not}

\subsection{Weakly singular integral equations}\label{sec-wseq}

In this section, we review terminology and basic theory of integral
equations, relevant to our problem.
We will be concerned with the Wiener--Hopf equations on the finite
interval $[0,T]$, $T<\infty$
%
\begin{equation}
\label{ueq} u(s,t) + \int_0^t u(r,t)
K(r,s)\,dr = f(s,t), \qquad0< s,t\le T,
\end{equation}
where the kernel $K:[0,T]^2\mapsto\Real$ and the forcing function
$f:[0,T]^2\mapsto\Real$ are given.

Note that the values of $u(s,t)$ on the sub-diagonal $\{0<s<t\le T\}$
determine $u(s,t)$ on the super-diagonal.
Hence, the problem of solving \eqref{ueq} reduces to solving it on the
sub-diagonal $\{0<s<t\}$ for all $t\in[0,T]$.
In this regard, \eqref{ueq} can be interpreted as an evolution
equation in the second (forward) variable.
Let us stress, however, that we will consider the solution $u$ as a
function on $[0,T]^2$.

For a fixed $t\in(0,T]$, the restriction of \eqref{ueq} to the sub-diagonal:
%
\begin{equation}
\label{ueqF} u(s,t) + \int_0^t u(r,t)
K(r,s)\,dr = f(s,t), \qquad0< s< t,
\end{equation}
is the Fredholm equation of the second kind, whose solvability is very
well-known under various conditions
(see, e.g., \cite{Kress}).

In this paper, we will consider weakly singular symmetric nonnegative
definite kernels satisfying \eqref{Kws}.
Iterates of $K$ are denoted by $K^{(m)}$:
\begin{eqnarray*}
K^{(1)}(s,t) &=&K(s,t),
\\
K^{(m)}(s,t) &=& \int_0^T
K^{(m-1)}(s,r)K(r,t)\,dr, \qquad m =2,3,\ldots.
\end{eqnarray*}
Recall that for $0<\alpha,\beta<1$
\[
\int_0^T \llvert s-r\rrvert^{-\alpha}
\llvert r-t\rrvert^{-\beta}\,dr \le
\cases{
\displaystyle C_1 \llvert s-t \rrvert^{1-\alpha-\beta}, &\quad$\alpha+\beta>1$,
\cr
\displaystyle C_2\log \frac{1}{\llvert s-t\rrvert} + C_3, &\quad$\alpha+\beta=1$,
\cr
C_4, &\quad$\alpha+\beta<1$,} %
\]
where $C_i$'s are constants. Therefore, singularity improves with
iterations and eventually disappears.

For weakly singular kernels, equation \eqref{ueqF} is uniquely
solvable in $L^1([0,t])$, provided $f(\cdot, t)\in L^1([0,t])$.
If $f(\cdot,t)$ is bounded on $[0,t]$, so is the solution $u(\cdot,t)$.
%
%
If the kernel $K$ is continuous outside the diagonal and $f(\cdot
,t)\in C([0,t])$, the solution $u(\cdot, t)$ is continuous on $[0,t]$,
but typically its derivative has discontinuities at the end points of
the interval.
The comprehensive accounts of these results can be found in \cite{Kress,VP80} and \cite{V93}.

The following lemma shows that the solution of \eqref{ueq} has at most
the same type of singularity as the forcing function.

\begin{lem}\label{lem-sing}
Assume that $\llvert f(s,t)\rrvert\le c \llvert s-t\rrvert^{-\beta
}$ with constants $c$ and
$\beta\in[0,1)$, then the solution of \eqref{ueqF} satisfies
\[
\bigl\llvert u(s,t)\bigr\rrvert\le C \llvert s-t\rrvert^{-\beta
},\qquad
s,t\in[0,T] %
\]
for some constant $C$.
\end{lem}

\begin{pf}
Let $m_0$ be an integer, such that $ \tilde f(s,t)=\int_0^t K^{(m_0)}
(s,r)f(r,t)\,dr$ is
bounded. The function
\[
\tilde u(s,t)=u(s,t) - f(s,t) - \sum_{m=1}^{m_0-1}
\int_0^t (-1)^m K^{(m)}
(s,r)f(r,t)\,dr, %
\]
solves the equation
%
\begin{equation}
\label{tildeueq} \tilde u(s,t)+ \int_0^t
\tilde u(r,t) K(r,s)\,dr = (-1)^{m_0}\tilde f(s,t), \qquad0\le s\le
t.
\end{equation}
Multiplying this equation by $\tilde u(s,t)$, integrating and using
positive definiteness of the kernel $K$, we get
\[
\int_0^t \tilde u^2(s,t) \,ds
\le\int_0^t \tilde u(s,t) \tilde f(s,t)\,ds, %
\]
and, by the Cauchy--Schwarz inequality,
%
\begin{equation}
\label{tildeuB} \int_0^t \tilde u^2(s,t) \,ds \le\int_0^t
\tilde f^2(s,t)\,ds\le\llVert\tilde f \rrVert
_\infty^2.
\end{equation}
Let $n_0$ be the integer such that $K^{(n_0)}$ is bounded,
then iterating \eqref{tildeueq} $n_0$ times gives
\begin{eqnarray*}
\tilde u(s,t) &=& \tilde f(s,t) + \sum_{m=1}^{n_0-1}
(-1)^m \int_0^t
K^{(m)}(r,s)\tilde f(r,t)\,dr
\\
&&{}+
(-1)^{n_0} \int_0^t
K^{(n_0)}(r,s)\tilde u(r,t)\,dr.
\end{eqnarray*}
The first two terms in the right-hand side are bounded, since
$\tilde f$
is bounded and
\[
\sup_{s\le T}\int_0^T
K^{(m)}(s,t)\,dt<\infty \qquad\forall m\ge1. %
\]
The last term is bounded due to \eqref{tildeuB}. It follows that
$\tilde u(s,t)$ is bounded
and therefore $\llvert u(s,t)\rrvert\le C_1\llvert s-t\rrvert
^{-\beta}$ for all $s< t\le T$ with
a constant $C_1$.
As discussed above, the solution of \eqref{ueq} on the super-diagonal
is determined by the solution on the sub-diagonal
and, therefore, the same bound holds for $t< s\le T$
possibly with a different constant.
\end{pf}

The following lemma shows that certain integrals of the solution are
determined by its values on the diagonal.

\begin{lem}\label{lem-qdiag}
Assume $f\in C([0,T])$ does not depend on $t$ and the partial
derivative $\dot u(s,t):=\frac{\partial}{\partial t}u(s,t)$ exists and
$\dot u(\cdot,t)\in L^1([0,t])$, then
\[
\int_0^t u(s,t)f(s) \,ds = \int
_0^t u^2(s,s) \,ds. %
\]
\end{lem}

\begin{pf}
Multiplying \eqref{ueqF} by $u(s,t)$ and integrating, we get
%
\begin{equation}
\label{intintfla} \qquad\int_0^t u^2(s,t)\,ds
+ \int_0^t \int_0^t
u(r,t) u(s,t)K(r,s)\,dr \,ds= \int_0^t
u(s,t)f(s)\,ds
\end{equation}
and, consequently,
\begin{eqnarray*}
\frac{d} {dt}\int_0^t u(s,t)f(s)\,ds &=&
u^2(t,t) + 2u(t,t)\int_0^t u(r,t)
K(r,t)\,dr
\\
&&{} +
2\int_0^t \dot u(r,t) \biggl( u(r,t) +
\int_0^t u(s,t)K(r,s) \,ds \biggr)\,dr
\\
&=&
u^2(t,t) + 2u(t,t) \bigl(f(t)-u(t,t) \bigr) + 2\int
_0^t \dot u(r,t)f(r) \,dr
\\
&=&
-u^2(t,t) + 2\frac{d} {dt}\int_0^t
u(r,t)f(r) \,dr,
\end{eqnarray*}
which gives the claimed identity.
\end{pf}


\subsubsection{Krein's method}

For kernels with certain special structure, the solution of \eqref
{ueqF} with an arbitrary forcing can be expressed in terms of its
solution with the
unit forcing. The following theorem is an adaptation of Theorem 8.1,
Section~8, Chapter IV in \cite{GK70}.

\begin{teo}\label{Kthm}
Assume that the equation
%
\begin{equation}
\label{Keqg} g(s,t)+ \int_0^t
g(r,t)K(r,s)\,dr = 1, \qquad0\le s\le t\le T,
\end{equation}
has a unique continuous solution and $g(t,t)\ne0$, $t\in[0,T]$. Then
the solution of~\eqref{ueqF} with an arbitrary $f(\cdot,t)\in L^1([0,T])$
is given by
%
\begin{equation}
\label{qviag} u(s,t) = g(s,t) F(t,t) -\int_s^t
g(s,u) \frac{\partial}{\partial u} F(u,t)\,du,
\end{equation}
where
%
\begin{equation}
\label{KrF} F(\tau,t) = \frac{1} {g^2(\tau,\tau)}\frac{\partial
}{\partial\tau} \int
_0^\tau g(s,\tau) f(s,t)\,ds.
\end{equation}
\end{teo}

\begin{rem}
The result in \cite{GK70} requires that the forcing function $f$ does
not depend on $t$ and is continuous.
The extension to integrable $f$ can be carried out through
approximation of $f$ by continuous functions in $L^1([0,t])$ in the
usual way.
The formulas \eqref{qviag} and \eqref{KrF}, where $f$ is allowed to
depend on $t$, are obtained by treating $t$ in the right-hand side of
\eqref{ueqF}
as a fixed parameter, applying the original formula (8.7) in \cite
{GK70} and then equating $t$ to the integration limit. We omit lengthy, but
otherwise routine details.
\end{rem}

The following class of kernels will be particularly useful for our purposes.

\begin{lem}\label{lem-dif}
Assume that $f$ does not depend on $t$, $f\in C_1 ((0,T) )\cap
C([0,T])$ and the kernel $K$ has the form
%
\begin{equation}
\label{Kchi} K(s,t) = \chi(s/t)\llvert s-t\rrvert^{-\alpha}, \qquad0\le
\alpha<1
\end{equation}
with $\chi\in C([0,\infty))$. Then the solution $u(s,t)$ of \eqref
{ueq} satisfies the following properties:
\begin{longlist}[(iii)]
\item[(i)]  $u(s,t)$ is\vspace*{1pt} continuously differentiable in $t\in(0,T]$
for any $s>0, s\ne t$, 

\item[(ii)]  the derivative
$\dot u(s,t):= \frac{\partial}{\partial t}u(s,t)$ solves the equation 
%
\begin{equation}
\label{WHdot}
\qquad \dot u(s,t) + \int_0^t \dot
u(r,t) K (r,s) \,dr = - u(t,t) K (s,t), \qquad0<s<t\le T
\end{equation}
and satisfies the bound
\[
\bigl\llvert\dot u(s,t)\bigr\rrvert\le C\llvert s-t\rrvert^{-\alpha},
\qquad0\le s,t\le T. %
\]

\item[(iii)] $\dot u(\cdot,t)\in L^2([0,t])$ for $\alpha< \frac{1}2$.\label{p3}
\end{longlist}
\end{lem}

\begin{pf}
(i) 
The function $u_t(x):= u(xt,t)$, $x\in[0,1]$, $t> 0$ satisfies the
integral equation
\[
u_t(x) + t^{1-\alpha}\int_0^1
u_t(y) K(x,y) \,dv=f(xt), \qquad u\in[0,1]. %
\]
As mentioned above, the unique continuous solution exists and
in the terminology of \cite{RN55}, any point $\lambda:=t^{1-\alpha}$
is regular.
The operator associated with the weakly singular kernel $K$ maps
$L^2([0,1])$ into
itself (see, e.g., Theorem 9.5.1 in \cite{EDV}).
It follows from, for example, the theorem on page 154
in \cite{RN55}, that the resolvent is analytic in $\lambda$, and
hence $u_t(x)$ is continuously differentiable at $t\in(0,T]$ for all
$x\in[0,1]$.
Differentiability of $u_t(x)$ with respect to $x\in(0,1)$ for
continuous $\chi(\cdot)$ can be shown by the method from \cite
{VP80}, using the
particular form of the kernel $K$.
Therefore, the function $u(s,t)= u_{t}(s/t)$ is continuously
differentiable at $t>0$ for any $s\in(0,t)$ and, therefore,
for $s\in(t,T]$ as well.

(ii)~
Equation \eqref{WHdot} is obtained by taking the derivative of both
sides of \eqref{ueq} and the bound follows from
Lemma~\ref{lem-sing}.

(iii)~
Obvious in view of (ii). 
\end{pf}

The crucial assumption in Theorem~\ref{Kthm}, inherited from Theorem
8.1 in \cite{GK70}, is that the solution of \eqref{Keqg} does not
vanish in the diagonal.
This property is guaranteed for symmetric difference kernels of the
form $K(s,t)=\kappa(s-t)$ with $\kappa\in L^1([-T,T])$
(Theorem 8.2 in \cite{GK70}).
Obviously such nondegeneracy cannot be expected to hold in general
(see Example~\ref{exinno}).
The following lemma extends applicability of Krein's method to kernels,
introduced in Lemma~\ref{lem-dif}:

\begin{lem}\label{cor}
The assertion of Theorem~\ref{Kthm} is true for kernels of the form~\eqref{Kchi}.
\end{lem}

\begin{pf}
We will argue that $g(t,t)\ne0$ by contradiction. Suppose $g(t,t)=0$
for some $t>0$.
Changing the integration variable, equation \eqref{Keqg} can be
rewritten as
\[
g(s,t) + s^{1-\alpha}\int_0^{t/s} g(xs,t)
\llvert1-x\rrvert^{-\alpha} \chi(x)\,dx = 1. %
\]
Taking the derivative of both sides with respect to $s$ and multiplying
by $s$, we get
\begin{eqnarray*}
&& sg'(s,t) + s^{2-\alpha}\int_0^{t/s}
x g'(xs,t) \llvert1-x\rrvert^{-\alpha
} \chi(x)\,dx
\\
&&\qquad =
-(1-\alpha)s^{1 -\alpha}\int_0^{t/s} g(xs,t)
\llvert1-x\rrvert^{-\alpha} \chi(x)\,dx,
\end{eqnarray*}
where $g'(s,t)=\frac{\partial}{\partial s}g(s,t)$ and we used $g(t,t)=0$.
Now change the variables back to get
\[
sg'(s,t) + \int_0^{t} r
g'(r,t) K(r,s)\,dr= -(1-\alpha) \bigl(1-g(s,t) \bigr). %
\]
Multiplying by $g(s,t)$ and integrating gives
\begin{eqnarray*}
&& -(1-\alpha) \int_0^t g(s,t)\,ds+(1-\alpha)
\int_0^t g^2(s,t)\,ds
\\
&&\qquad =
\int_0^t sg'(s,t)g(s,t)\,ds +
\int_0^t\int_0^t
r g'(r,t) g(s,t) K(r,s) \,dr\,ds
\\
&&\qquad =
\int_0^t r g'(r,t)\,dr =-\int
_0^t g(r,t)\,dr
\end{eqnarray*}
and, after a rearrangement,
\[
(1-\alpha)\int_0^t g^2(s,t)\,ds+
\alpha\int_0^t g(s,t)\,ds =0. %
\]
By Lemma~\ref{lem-qdiag}, it follows that
\[
(1-\alpha)\int_0^t g^2(s,t)\,ds+
\alpha\int_0^t g^2(s,s)\,ds =0.
\]
This implies that $g(s,t)= 0$ for a.e. $s\in[0,t]$, which contradicts
\eqref{Keqg}.
\end{pf}

\begin{cor}\label{corf}
For the kernel \eqref{Kchi} and $f(s)=s^{\beta}$
with $\beta\ge0$, the solution of \eqref{ueq} does not vanish on the
diagonal, that is, $u(t,t)\ne0$ for all $t\in(0,T]$.
\end{cor}
\begin{pf}
The function $\tilde u(s,t) :=s^{-\beta}u(s,t)$ solves the equation
\[
\tilde u(s,t) + \int_0^t \tilde u(r,t)
(r/s)^{\beta}K(r,s)\,dr = 1. %
\]
The claim follows, since the kernel $(r/s)^{\beta}K(r,s)$ satisfies
the assumption of Lemma~\ref{cor}.
\end{pf}

Krein's method reveals yet another useful formula.

\begin{cor}\label{lem-R}
The function
$
L(s,t) = \frac{\dot g(s,t)}{g(t,t)}
$
satisfies the equation \eqref{Leq}
and
\[
L(s,t) - L(t,s) = \int_s^t L(s,\tau)L(t,\tau)
\,d\tau, \qquad s<t. %
\]
\end{cor}

\begin{pf}
Equation \eqref{Leq} readily follows from Lemma~\ref{lem-dif}.
By \eqref{KrF},
\begin{eqnarray*}
F(\tau,t) &=& -\frac{1}{g^2(\tau,\tau)}\frac{\partial}{\partial
\tau}\int_0^\tau
g(r,\tau)K(r,t)\,dr
\\
&=&
-\frac{1}{g^2(\tau,\tau)} \frac{\partial}{\partial\tau} \bigl(
1-g(t,\tau) \bigr) =
\frac{\dot g(t,\tau)}{g^2(\tau,\tau)},
\end{eqnarray*}
and, integrating by parts in \eqref{qviag}, we get
\begin{eqnarray*}
L(s,t) &=& \frac{\dot g(t,s)}{g(s,s)} +\int_s^t
\frac{\dot g(t,\tau)}{g^2(\tau,\tau)} \dot g(s,\tau)\,d\tau
\\
&=& L(t,s) +\int_s^t
L(s,\tau)L(t,\tau) \,d\tau. %
\end{eqnarray*}\upqed
\end{pf}

\subsection{Stochastic integral representation}

Consider the process $X$ defined in~\eqref{XBG} and let $\eta$ be a
random variable, such that the pair $(\eta,X_t)$ forms a Gaussian process.
Then $\E(\eta\mid\F^X_t)$ belongs to the closure $\mathcal{H}_t^X$
of the linear combinations of the increments of $X$ in $L^2(\Omega,\P
)$. What is less apparent is that
this conditional expectation can be expressed as a stochastic integral
with respect to $X$.
For example, it is not hard to see that such a representation is
impossible in Example~\ref{ex}.

We will assume that the stochastic integral with respect to the process
$G$ is defined on a scalar product space $\Lambda_t$ of
deterministic functions $f:[0,t]\mapsto\Real$, in which simple
(piecewise constant) functions are dense. For $f\in\Lambda_t$
\[
\int_0^t f(s)\,dG_s = \lim
_{n\to\infty} \int_0^t
f_n(s)\,dG_s\qquad\mbox{in } L^2(\Omega, \P),
\]
whenever $f_n\to f$ in $\Lambda_t$. Also we will assume that
%
\begin{equation}
\label{GLambda} \E\int_0^t f(s)
\,dG_s \int_0^t h(s)
\,dG_s = \langle f,g\rangle_{\Lambda_t},
\end{equation}
and either $L^2([0,t])\subseteq\Lambda_t$ or $\Lambda_t$ is complete.

All these assumptions hold for a variety of familiar processes,
including fBm. Let us stress, however, that they do not exclude the
possibility of $\mathcal{H}^G_t$ being strictly larger than the image
of $\Lambda_t$ in $L^2(\Omega,\P)$ under the stochastic integral.
For example, not all random variables in $\mathcal{H}^{B^H}_t$ can be
expressed as stochastic integrals
with respect to $B^H$ (see \cite{PT01}).

\begin{lem}\label{lem-rep}
Under the above assumptions,
\[
\E\bigl(\eta\mid\F^X_t\bigr) = \E\eta+ \int
_0^t h(s,t)\,dX_t, %
\]
with a unique function $h(\cdot,t)\in L^2([0,t])\cap\Lambda_t$.
\end{lem}

\begin{pf}
Following the arguments of the proof of Lemma 10.1 in \cite{LS1}, let
$t_i = t i/ 2^{n}$, $i=0,\ldots,2^n$ and $\F^X_{t,n}=\sigma\{
X_{t_{i}}-X_{t_{i-1}}, i=1,\ldots,2^n\}$.
Then $\F^X_{t,n}\nearrow\F^X_t$ and by the martingale convergence
%
\begin{equation}
\label{lim} \lim_n \E\bigl(\eta\mid
\F^X_{t,n}\bigr) = \E\bigl(\eta\mid\F^X_t
\bigr)\qquad\mbox{in } L^2(\Omega,\P).
\end{equation}
By the normal correlation theorem,
\[
\E\bigl(\eta\mid\F^X_{t,n}\bigr) = \E\eta+ \sum
_{i=1}^{2^n} h_{i-1}^n
(X_{t_i}-X_{t_{i-1}} ), %
\]
with constants $h_{i-1}^n$, $i=1,\ldots,2^{n}$. Define
\[
h_n(s,t):= \sum_{i=1}^{2^n}
h_{i-1}^n \one{s\in[t_{i-1},t_i)},
\]
then
\[
\E\bigl(\eta\mid\F^X_{t,n} \bigr) = \E\eta+ \int
_0^t h_n(s,t)\,dB_s +
\int_0^t h_n(s,t)\,dG_s,
\]
and by independence of $B$ and $G$
\[
\E\bigl(\E\bigl(\eta\mid\F^X_{t,n} \bigr)- \E\bigl(\eta
\mid\F^X_{t,m} \bigr) \bigr)^2 = \llVert
h_n-h_m\rrVert_2^2 + \llVert
h_n-h_m\rrVert_{\Lambda_t}^2.
\]
Therefore, by \eqref{lim}
\begin{eqnarray*}
&& \lim_n \sup_{m\ge n} \bigl( \llVert
h_n-h_m\rrVert^2_{2} +\llVert
h_n-h_m\rrVert^2_{\Lambda
_t} \bigr)
\\
&&\qquad =
\lim_n \sup_{m\ge n}\E\bigl(\E\bigl(\eta
\mid\F^X_{t,n} \bigr)-\E\bigl(\eta\mid\F^X_{t,m}
\bigr) \bigr)^2=0,
\end{eqnarray*}
and hence $h_n\to h$ in $L^2([0,t])$ by its completeness. Since we
assumed that either $L^2([0,t])\subseteq\Lambda_t$ or $\Lambda_t$ is
complete,
$h\in\Lambda_t$ as well.
The claimed representation now follows, since
%
\begin{eqnarray*}
&& \E\biggl(\E\bigl(\eta\mid\F^X_t\bigr)-\E\eta- \int
_0^t h(s,t)\,dB_s -\int
_0^t h(s,t)\,dB^H_s
\biggr)^2
\\
&&\qquad \le
3\E\bigl(\E\bigl(\eta\mid\F^X_t\bigr)-\E\bigl(\eta
\mid\F^X_{t,n}\bigr) \bigr)^2 + 3 \llVert
h-h_n\rrVert^2_2 + 3 \llVert
h-h_n\rrVert^2_{\Lambda_t} \xrightarrow{n\to\infty}
0.
\end{eqnarray*}
The uniqueness of $h$ is obvious.
\end{pf}

\begin{rem}\label{rem-Xrep}
In particular, the conclusion of Lemma~\ref{lem-rep} holds, when $G$
satisfies \eqref{GK} with $K$, such that
\[
\sup_{s\le T} \int_0^T \bigl
\llvert K(s,t)\bigr\rrvert \,dt<\infty. %
\]
In this case, \eqref{GLambda} holds with
\[
\langle f,g\rangle_{\Lambda_t} =\int_0^t
\int_0^t f(s)g(s) K(s,r)\,ds\,dr, %
\]
and $L^2([0,t])\subseteq\Lambda_t$.
Note that weakly singular kernels as in \eqref{Kws} fit this framework.
\end{rem}

\subsection{Fractional Brownian motion}

In the proofs below, we will frequently use a number of well-known
formulas, related to fBm.
Our main references are \cite{NVV99} and \cite{PT01}.

\subsubsection{Constants}
%
\begin{eqnarray}\label{constants}
c_H &=& \frac{1}{2H\Gamma(\sfrac{3} 2-H )\Gamma(H+\sfrac{1}
2 )}, \nonumber
\\
\lambda_H &=& \frac{2H\Gamma(H+\sfrac{1} 2
)\Gamma(3-2H)}{ \Gamma(\sfrac{3} 2-H )},
\\
\beta_H&=& c^2_H \biggl(\frac{1} 2-H
\biggr)^2 \frac{\lambda_H }{2-2H}.\nonumber
\end{eqnarray}

\subsubsection{Spaces and operators}
For $f:[0,t]\mapsto\Real$, define the operators
%
\begin{equation}
\label{Psi}
\qquad (\Psi f) (s,t)= -2H\frac{d}{ds}\int_s^t
f(r) r^{H-\sfrac{1}2}(r-s)^{H-\sfrac{1}2} \,dr ,\qquad0\leq s\leq t ,
\end{equation}
and
\[
(\Phi f) (s)= \frac{d}{ds}\int_0^s
f(r)r^{\sfrac{1}2-H} (s-r)^{\sfrac{1}2-H} \,dr, \qquad0\leq s\leq t. %
\]
These formulas can be readily expressed in terms of the
Riemann--Liouville fractional integrals and derivatives; see \cite{PT01}.
The respective inverse operators are given by
%
\begin{equation}
\label{invK} \bigl(\Psi^{-1} g\bigr) (s,t)= - c_H
s^{1/2-H} \frac{d}{ds} \int_s^t
g(r,t) (r-s)^{1/2-H} \,dr
\end{equation}
and
%
\begin{equation}
\label{invQ} \bigl(\Phi^{-1} g\bigr) (s) = 2H c_H
s^{H-1/2}\frac{d}{ds}\int_0^s
g(r) (s-r)^{H-1/2}\,dr.
\end{equation}
Let us define the space
\[
\Lambda_t^{H-1/2} : = \biggl\{f: [0,t]\mapsto\Real\mbox{ such that } \int_0^t \bigl(s^{1/2-H}(
\Psi f) (s,t) \bigr)^2\,ds <\infty\biggr\}, %
\]
with the scalar product
%
\begin{equation}
\label{sprod} \langle f,g\rangle_{\Lambda_t^{H-1/2}} := \frac
{2-2H}{\lambda_{H}} \int
_0^t s^{1-2H}(\Psi f) (s,t) (\Psi g)
(s,t) \,ds.
\end{equation}
The inclusion $L^2([0,t])\subset\Lambda^{H-1/2}_t$ holds for
$H>\frac{1}2$ and fails for $H<\frac{1} 2$ (see Remark 4.2 in \cite{PT01}).

The expression in \eqref{sprod} can be rewritten as
%
\begin{equation}
\label{phipsifla}
\qquad \langle f,g\rangle_{\Lambda_t^{H-1/2}}= H\int_0^tf(r)
\frac{d}{dr} \int_0^t g (u) \llvert r-u
\rrvert^{2H-1} \sign(r-u)\,du \,dr,
\end{equation}
which for $H>\frac{1} 2$ becomes
\[
\langle f,g\rangle_{\Lambda_t^{H-1/2}} = \int_0^t
\int_0^s K_H(u,v)f(u)g(v)\,du\,dv,
\]
with kernel $K_H$ defined in \eqref{K}.

Finally, for any $H\in(0,1)$ and $f,g\in L^2([0,t])\cap\Lambda
_t^{H-1/2}$, the following identity holds:
%
\begin{equation}
\label{inner_prod} \int_0^tf(s)g(s)
\,ds=c_H\int_0^t(\Psi f) (s,t) (
\Phi g) (s) \,ds.
\end{equation}

\subsubsection{Representation formulas for fBm}
The stochastic integral with respect to fBm, $\int_0^t f(s)\,dB^H_s$,
can be defined for $f\in\Lambda^{H-1/2}_t$ in the usual way
(see \cite{PT01}).
Integrals with kernels
%
\begin{eqnarray}
\label{rhorho} \rho(s,t) &=& \sqrt{\frac{2-2H}{\lambda_H}} s^{1/2-H} (\Psi1)
(s,t),
\nonumber\\[-8pt]\\[-8pt]\nonumber
\tilde\rho(s,t) &=& \sqrt{\frac{2-2H}{\lambda_H}} \bigl(\Psi^{-1}
u^{H-1/2}\bigr) (s,t),
\end{eqnarray}
transform fBm into standard Brownian motion and vise versa. Namely,
\[
B^H_t = \int_0^t
\rho(s,t)\,dW_s, %
\]
where
\[
W_t = \int_0^t \tilde\rho(s,t)\,dB^H_s %
\]
is a standard Brownian motion. More generally,
\[
\int_0^t f(s)\,dB^H_s
= \int_0^t \sqrt{\frac{2-2H}{\lambda_H}}
s^{1/2-H} (\Psi f) (s,t) \,dW_s %
\]
and
\[
\int_0^t f(s)\,dW_s = \int
_0^t \sqrt{\frac{2-2H}{\lambda_H}} \bigl(
\Psi^{-1} u^{H-1/2}f(u) \bigr) (s,t) \,dB^H_s.
\]
It follows that
\[
\E\int_0^t f(u)\,dB^H_u
\int_0^s g(v)\,dB^H_v
= \langle f,g\rangle_{\Lambda_t^{H-1/2}}. %
\]

\section{Proof of Theorem \texorpdfstring{\protect\ref{teo2}}{2.2}}\label{sec-thm2}

As mentioned in Section~\ref{sec-wseq}, for kernels of the form \eqref
{Kws}, equation \eqref{Leq} has a unique solution $L(\cdot,t)\in L^1([0,t])$.
Let
\[
W_t = \int_0^t
\phi(s)\,dB_s, %
\]
with
$ \phi(s) = 1-\int_0^s L(r, s)\,dr$, and define the martingale
$
\overline{B}_t := \E(W_t\mid\F^X_t )$.
By Lemma~\ref{lem-rep} (see Remark~\ref{rem-Xrep})
\[
\overline{B}_t = \int_0^t q(s,t)
\,dX_s, %
\]
where $q$ solves the integral equation \eqref{Qeq}.
A direct substitution shows that the unique solution is given by
\[
q(s,t) = 1 + \int_s^t L(s,r)\,dr, %
\]
and hence
\[
\langle\overline{B}\rangle_t = \E\overline{B}_t^2
= \E\overline{B}_t W_t = \int_0^t
q(s,t) \phi(s)\,ds \stackrel{\dagger} {=} \int_0^t
q^2(s,s)\,ds = t, %
\]
where the equality $\dagger$ holds by Lemma~\ref{lem-qdiag} and the
last equality holds since \mbox{$q(s,s)=1$}.
By the L\'{e}vy theorem, $\overline{B}$ is a Brownian motion in
filtration $\F^X_t$.

Let us check that $\F^{\overline{B}}_t = \F^X_t$. Since $\F
^{\overline{B}}_t\subseteq\F^X_t$, the process
$D_t = X_t -\E(X_t\mid\F^{\overline{B}}_t)$ is $\F^X_t$-adapted, and
hence admits
the representation
\[
D_t = \int_0^t
h(s,t)\,dX_s, %
\]
for some $h(\cdot,t)\in L^2([0,T])$. On the other hand, by the
orthogonality property of conditional expectation, $\E D_t \overline
{B}_s=0$ for all $s\le t$.
Let us show that this condition implies $h(s,t)= 0$ for all $s\le t$,
that is, $D_t= 0$. To this end, we have
\begin{eqnarray*}
\E D_t \overline{B}_s &=& \E\int_0^t
h(u,t) \,dX_u \int_0^s
q(u,s)\,dX_u
\\
&=&
\int_0^s h(u,t) q(u,s)\,du + \int
_0^t h(u,t) \int_0^s
q(v,s) K(v,u)\,dv \,du
\\
&=&
\int_0^s h(u,t) q(u,s)\,du + \int
_0^t h(u,t) \bigl(\varphi(u)-q(u,s) \bigr) \,du
\\
&=&
\int_0^t h(u,t) \varphi(u)\,du -\int
_s^t h(u,t) q(u,s)\,du,
\end{eqnarray*}
where we used \eqref{Qeq}. Since $\E D_t \overline{B}_s=0$ for all
$s\le t$, taking the derivative with respect to $s$, we obtain the
Volterra equation for $h(s,t)$:
%
\begin{equation}
\label{LVolt} h(s,t) - \int_s^t h(u,t)
L(u,s)\,du =0, \qquad s\le t.
\end{equation}
Recall that $L$ solves equation \eqref{Leq}, and thus by Lemma~\ref
{lem-sing}, $\llvert L(s,t)\rrvert\le C\llvert s-t\rrvert^{-\alpha
}$ with a constant $C$.
Iterating \eqref{LVolt} sufficient number of times, it follows that
$h$ also solves the Volterra equation with a
bounded kernel.
Such equations are well known to have a unique solution, and hence
$h(s,t)\equiv0$. Consequently,
$X_t = \E(X_t\mid\F^{\overline{B}}_t)$ and $\F^{\overline{B}}_t=\F^X_t$.

Finally, the inverse transformation \eqref{XqB} 
follows from the normal correlation theorem:
\begin{eqnarray*}
\hat q(s,t) &=& \frac{\partial}{\partial s} \E X_t \overline{B}_s
= \frac{\partial}{\partial s} \E X_t \int_0^s
q(r,s)\,dX_r
\\
&=&
\frac{\partial}{\partial s} \biggl( \int_0^s q(r,s)
\,dr + \int_0^t \int_0^s
K(r,u)q(r,s)\,dr\,du \biggr)
\\
&=&
\frac{\partial}{\partial s} \biggl( \int_0^s q(r,s)
\,dr + \int_0^t \bigl(\phi(u)-q(u,s) \bigr)\,du
\biggr)
\\
&=& -\frac{\partial}{\partial s} \int_s^tq(r,s) \,dr.
\end{eqnarray*}


\section{Proof of Theorem \texorpdfstring{\protect\ref{teo1}}{2.4}}\label{sec-thm1}

Let us first sketch the main steps of the proof. Our candidate for the
innovation process is the martingale
$M_t = \E(B_t\mid\F^X_t)$.
We argue that it can be represented as a stochastic integral with
respect to $X$ and then identify the integrand
as the solution of equation \eqref{ideq}. This equation is solved by
reduction to a weakly singular integral equation, whose precise form
depends on whether $H$ is greater or less than $\frac{1} 2$. The
quadratic variation $\langle M\rangle_t$ is then expressed as an
integral of the solution on
the diagonal, which gives the formula \eqref{Mqv}.

The claimed assertions are obvious for the case $H=\frac{1} 2$, which we
exclude from the consideration thereafter.

\subsection{Proof of part \textup{(i)}}

\subsubsection{The equation \protect\eqref{ideq} and its alternative forms}

The following theorem proves part (i) of Theorem~\ref{teo1} and
elaborates on the structure of equation \eqref{ideq}.

\begin{teo}\label{teo-eq} The representation \eqref{MX} holds with
$g(s,t)$, $s\le t$, being the
unique continuous solution of
the following equations:
\begin{longlist}[(iii)]
\item[(i)] for $H\in(0,1]$, the integro-differential equation~\eqref{ideq},

\item[(ii)] for $H\in(0,1)$,
the fractional integro-differential equation
%
\begin{equation}
\label{f_i_e} c_H (\Phi g) (s) + \frac{2-2H}{\lambda_{H}}(\Psi g)
(s,t)s^{1-2H}=c_H(\Phi1) (s), \qquad s\in(0,t],
\end{equation}

\item[(iii)] for $H\in(\frac{1} 2,1]$, the weakly singular
integral equation
%
\begin{equation}
\label{WHt} g(s,t) + \int_0^t g(r,t)
K_H(r,s) \,dr = 1, \qquad s\in(0,t],
\end{equation}
with kernel $K_H$, defined in \eqref{K},

\item[(iv)]  for $H\in(0,\frac{1} 2)$, the weakly singular
integral equation
%
\begin{eqnarray}
\label{WH_smallH}
&& g(s,t) + \beta_Ht^{-2H}\int
_0^t g(r,t)K^H \biggl(
\frac{r}{t},\frac
{s}{t} \biggr) \,dr
\nonumber\\[-8pt]\\[-8pt]\nonumber
&&\qquad =
c_H s^{\sfrac{1}2-H} (t-s)^{\sfrac{1}2-H}, \qquad s\in[0,t],
\end{eqnarray}
with the kernel
%
\begin{equation}
\label{kappabar} K^H(u,v)= \llvert u-v\rrvert^{-2H} N (u,v ),
\end{equation}
where $N\in C([0,1]^2)$ is given in \eqref{Nfla} below.
\end{longlist}
\end{teo}

\begin{pf}
By Lemma~\ref{lem-rep}, there exists a function $g(\cdot,t)\in
L^2([0,t])\cap\Lambda_t^{H-1/2}$,
such that
\[
M_t =\E\bigl(B_t\mid\F^X_t\bigr)
= \int_0^t g(s,t)\,dX_s.
\]
To verify the representation \eqref{MX}, we have to check that
$g(s,t)$ uniquely solves each one of the equations in
(i)--(iv). To this end, we will argue that $g(s,t)$
satisfies the equation from (ii) for almost every
$s\in[0,t]$. Then we show that this equation reduces to (iii)
for $H>\frac{1} 2$ and to (iv) for $H<\frac{1} 2$.
These weakly singular integral equations are well known to have unique
continuous solutions and, therefore, $g(s,t)$, in fact, satisfies
(ii) for all $s\in[0,t]$.
Finally, we will argue that (ii) and (i) share the
same solution.

For any test function $\varphi\in L^2([0,t])\cap\Lambda
_t^{H-1/2}$, the orthogonality
property of the conditional expectation implies
%
\begin{eqnarray}\label{ortprop}
0 &=& \E\biggl(B_t-\int
_0^t g(s,t)\,dX_s \biggr)\int
_0^t \varphi(s)\,dX_s \nonumber
\\
&=&
\int_0^t
\varphi(s)\,ds - \int_0^t \varphi(s) g(s,t)\,ds \nonumber
\\
&&{} -
\frac{2-2H}{\lambda_H}\int_0^t s^{1-2H} (
\Psi g) (s,t) (\Psi\varphi) (s,t) \,ds
\\
&=&
\int_0^t (\Psi\varphi) (s,t) \biggl(
c_H (\Phi1) (s)\,ds - c_H (\Phi g) (s,t) \nonumber
\\
&&{}-
\frac{2-2H}{\lambda_H} s^{1-2H} (\Psi g) (s,t) \biggr) \,ds,\nonumber
\end{eqnarray}
where we used the identity \eqref{inner_prod}. Since $\varphi$ can be
an arbitrary differentiable function,
$g(s,t)$ satisfies \eqref{f_i_e} for almost all $s\in[0,t]$.

Applying the transformation \eqref{invQ} with $H>\frac{1} 2$ to
equation \eqref{f_i_e}, a direct calculation shows that $g(s,t)$
satisfies \eqref{WHt}.
This weakly singular equation is well known to have a unique solution
(see, e.g., \cite{VP80}), continuous on $[0,t]$.
Since the transformation \eqref{invQ} is invertible, $g(s,t)$ is also
the unique continuous solution of~\eqref{f_i_e} for $H>\frac{1} 2$.

%
%
%
%
%
%
%
%
Similarly, equation \eqref{WH_smallH} is obtained for $H<\frac{1} 2$ by
multiplying \eqref{f_i_e} by $\frac{\lambda_{H}}{2-2H}s^{2H-1}$ and
applying the transformation $\Psi^{-1}$.
A calculation shows that
\[
K^H(u,v)= (uv)^{\sfrac{1}2-H}\int_{u\vee v}^{1}r^{2H-1}
(r-u)^{-\sfrac{1}2-H}(r-v)^{-\sfrac{1}2-H} \,dr, %
\]
and changing the integration variable to $x:=\frac{1-v}{u-v}\frac
{r-u}{1-r}$ we get
\eqref{kappabar},
where
%
\begin{eqnarray}\label{Nfla}
N(u,v) &=& \biggl(\frac{a} b \biggr)^{1/2-H}
\nonumber\\[-8pt]\\[-8pt]\nonumber
&&{}\times \int
_0^\infty x^{-\sfrac{1} 2-H}(1+x)^{-\sfrac{1} 2-H}
\biggl(1+ \biggl(1-\frac{a} b \biggr)x \biggr)^{2H-1}\,dx,
\end{eqnarray}
with
\[
a = \frac{u} {1-u} \wedge\frac{v} {1-v}, \qquad b = \frac{u} {1-u}
\vee\frac{v} {1-v}. %
\]
For $H<\frac{1} 2$ the function $N$ is continuous and thus kernel $K^H$
is weakly singular. Since
the right-hand side of \eqref{WH_smallH} is a continuous function for
$H<\frac{1} 2$, this equation has a unique solution, continuous on $[0,t]$.
This completes the proof of (iv) and, in turn, of (ii).

Further, the identity \eqref{phipsifla} and orthogonality property
\eqref{ortprop} imply
\begin{eqnarray*}
0 &=& \int_0^t \varphi(s)\,ds - \int
_0^t \varphi(s) g(s,t)\,ds - \frac{2-2H}{\lambda_H}
\int_0^t s^{1-2H} (\Psi g) (s,t) (\Psi
\varphi) (s,t) \,ds
\\
&=&\int_0^t \varphi(s) \biggl( 1-g(s,t) - H
\frac{d}{ds} \int_0^t g(r,t) \llvert s-r
\rrvert^{2H-1} \sign(s-r)\,dr \biggr)\,ds.
\end{eqnarray*}
The assertion (i) follows, in view of arbitrariness of
$\varphi$ and unique solvability of~\eqref{f_i_e}.
Finally, for $t\in[0,T]$,
\[
\langle M \rangle_t = \E M_t^2 = \E
B_t M_t = \E B_t \int_0^t
g(s,t)\,dX_s = \int_0^t g(s,t)\,ds.
\]\upqed
\end{pf}

\subsection{Proof of part \textup{(ii)} for \texorpdfstring{$H>\frac12$}{H>1/2}}
Note that in this case, the derivative and integration in \eqref{Psi}
can be interchanged, and hence
\[
(\Psi g) (s,t)= 2H \biggl(H- \frac{1}2 \biggr)\int_s^t
g(r,t) r^{H-\sfrac{1}2}(r-s)^{H-3/2} \,dr ,\qquad0\leq s\leq t , %
\]
and $(\Psi g)(t,t)=0$ for all $t\in[0,T]$. Therefore, \eqref{Mqv}
holds by \eqref{MX} and Lemma~\ref{lem-qdiag},
and, in fact,
%
\begin{equation}
\label{MqvHlarge} \langle M\rangle_t = \int_0^t
g^2(s,s)\,ds>0.
\end{equation}
Since $g(t,t)>0$ for all $t\in[0,T]$ there exists a function $\hat
g(\cdot,t)\in L^2([0,t])$, such that
\[
\E\bigl(X_t\mid\F^M_t\bigr) = \int
_0^t \hat g(s,t)\,dM_s, \qquad t
\in[0,T]. %
\]
By the normal correlation theorem,
\[
\hat g(s,t)=\frac{d}{d\langle M \rangle_s}\E X_t M_s, %
\]
and the formula \eqref{tildeg} follows, since
%
\begin{eqnarray}\label{gen}
\E X_t M_s&=&
\int_0^sg(r,s)\frac{\partial}{\partial r}\E
X_tX_r \,dr\nonumber
\\
&=&
\int_0^s g(r,s)\,dr + H\int
_0^s g(r,s) \bigl(r^{2H-1} +
(t-r)^{2H-1} \bigr) \,dr
\nonumber\\[-8pt]\\[-8pt]\nonumber
&=&
\langle M \rangle_s + \int_0^t
H\frac{d}{d\tau}\int_0^s g(r,s) \llvert r-
\tau\rrvert^{2H-1}\sign(r-\tau)\,dr \,d\tau\nonumber
\\
&=&
\langle M \rangle_s + \int_0^t
\bigl(1-g(\tau,s)\bigr)\,d\tau.\nonumber
\end{eqnarray}
It is left to check that $\F^X_t= \F^M_t$, that is, $\E(X_t\mid\F
^M_t)=X_t$ or
%
\begin{equation}
\label{xixi} \E\bigl(X_t- \E\bigl(X_t\mid
\F^M_t \bigr) \bigr)^2=\E
X_t^2 -\E\bigl(\E\bigl(X_t\mid
\F^M_t \bigr) \bigr)^2 =0.
\end{equation}
Since $X_0=\E(X_0\mid\F^M_0)=0$, \eqref{xixi} holds if
\[
\frac{\partial^2}{\partial t\,\partial s} \int_0^{t\wedge s} \hat
g(r,t)\hat
g(r,s)\,d\langle M\rangle_r = K_H(t,s), \qquad s<t,
\]
or, by \eqref{MqvHlarge}, if
%
\begin{equation}
\label{dotGG} \dot{\hat g}(s,t) \hat g(s,s)g^2(s,s) + \int
_0^{s} \dot{\hat g} (r,t)\dot{\hat
g}(r,s)g^2(r,r)\,dr = K_H(t,s).
\end{equation}
Further, we have
\begin{eqnarray*}
\hat g(t,t) &= & 1 - \frac{1}{g^2(t,t)} \int_0^t
\dot g(s,t) \,ds = 1 - \frac{1}{g^2(t,t)} \biggl(\frac{d}{dt}\int
_0^t g(s,t)\,ds-g(t,t) \biggr)
\\
&=&
1 - \frac{1}{g^2(t,t)} \bigl(g^2(t,t)-g(t,t) \bigr) =
\frac{1} {g(t,t)},
\end{eqnarray*}
and, since $\dot{\hat g}(s,t) g(s,s) =-L(t,s)$, \eqref{dotGG} becomes
%
\begin{equation}
\label{RRR} -L(t,s) + \int_0^{s} L(t,r)
L(s,r) \,dr = K_H(t,s).
\end{equation}
Recall that the function $L$, satisfies equation \eqref{Leq}.
Rearranging the terms, multiplying by $L(s,u)$ and integrating gives
\begin{eqnarray*}
&& \int_0^s L(t,u) L(s,u)\,du + \int
_0^s K_H(t,u) L(s,u)\,du
\\
&&\qquad =
- \int_0^s \int_0^u
L(r,u)L(s,u) K_H(r,t) \,dr\,du
\\
&&\qquad =
- \int_0^s \biggl(\int
_r^s L(r,u)L(s,u)\,du \biggr) K_H(r,t)
\,dr
\\
&&\qquad =
- \int_0^s \bigl(L(r,s)-L(s,r) \bigr)
K_H(r,t) \,dr,
\end{eqnarray*}
where we used Corollary~\ref{lem-R}.
The second term on the left-hand side and the last term on the
right-hand side
cancel out and we get
\[
\int_0^s L(t,u) L(s,u)\,du=- \int
_0^s L(r,s) K_H(r,t) \,dr = L(t,s) +
K_H(s,t), %
\]
which verifies \eqref{RRR} and, therefore, \eqref{xixi}, thus
completing the proof. 

\subsection{An auxiliary process \texorpdfstring{$\widetilde X$}{widetilde X}}
The proof of (ii) of Theorem~\ref{teo1} in the case $H<\frac{1} 2$
relies on
the auxiliary process
%
\begin{equation}
\label{def_rho} \widetilde X_t = \int_0^t
\tilde{\rho}(s,t) \,dX_s,
\end{equation}
where the kernel $\tilde\rho$ is defined in \eqref{rhorho}. In
this section, we explore some of its properties.

\begin{lem}\label{lem-tildeX}
$\widetilde{X} = \widetilde B + \widetilde{U}$,
where $\widetilde B$ is a Brownian motion in its own filtration and
$\widetilde{U}$ is a centered Gaussian process with the covariance
function, satisfying
%
\begin{equation}
\label{tilde_kappa} \widetilde K_H (s,t) :=\frac{\partial^2}{\partial
s\,\partial t} \E
\widetilde{U}_s\widetilde{U}_t = \llvert t-s\rrvert
^{-2H}\chi(t/s ), \qquad s\ne t,
\end{equation}
where
\[
\chi(u)=\beta_H \bigl(u\wedge u^{-1} \bigr)^{\sfrac{1}2-H}L
\biggl( \frac
{1}{u\vee u^{-1}-1} \biggr),\qquad u\in\Real_+, %
\]
and
\[
L(v)=\int_{0}^{v}r^{-\sfrac{1}2-H}(1+r)^{-\sfrac{1}2-H}
\biggl(1-\frac
{r}{v} \biggr)^{1-2H} \,dr. %
\]
Moreover, $\F^X_t=\F^{\widetilde X}_t$ for all $t\in[0,T]$.
\end{lem}


\begin{pf}
It is well known (see, e.g., \cite{NVV99}),
that the integral transformation \eqref{def_rho} is
invertible:
\[
X_t = \int_0^t \rho(s,t) \,d
\widetilde X_s, \qquad t\in[0,T], %
\]
where $\rho$ is defined in \eqref{rhorho}.
In particular, $\F^X_t=\F^{\widetilde X}_t$, $t\in[0,T]$.

Further,\vspace*{1pt} it follows from \cite{NVV99} that the process $\widetilde B_t
= \int_0^t \tilde\rho(s,t)\,dB^H_s$ is the Brownian motion
in its own filtration.
Hence, $\widetilde X=\widetilde B + \widetilde U$ with
\[
\widetilde U_t = \int_0^t
\tilde\rho(s,t) \,dB_s, %
\]
and
\[
\widetilde K_H(s,t)=\frac{\partial^2}{\partial s\,\partial t}\E
\widetilde U_t
\widetilde U_s = \frac{\partial^2}{\partial s\,\partial
t} \int_0^{s\wedge t}
\tilde\rho(r,s)\tilde\rho(r,t)\,dr =\int_0^{s \wedge t}
\dot{\tilde\rho}(r,s)\dot{\tilde\rho}(r,t)\,dr, %
\]
where $\dot{\tilde\rho}(s,t) = \frac{\partial} {\partial t}
{\tilde\rho}(s,t)$ and we used the property $\tilde\rho(s,s)=0$.
The expression in~\eqref{tilde_kappa} is obtained by a direct
calculation, using the explicit expression \eqref{rhorho} for
$\tilde\rho(s,t)$.
\end{pf}


\subsection{The martingale $M$ for \texorpdfstring{$H<\frac12$}{H<1/2}}

The structure of the martingale $M$ and its relation to the process
$\widetilde X$ are
described in detail in the following lemma.

\begin{lem}\label{mart-brak}
For $H<\frac{1} 2$ and $t\in[0,T]$,
%
\begin{equation}
\label{MtildeX} M_t = \int_0^t
p(s,t) \,d\widetilde X_s,\qquad\langle M\rangle_t = \int
_0^t p^2(s,s)\,ds,
\end{equation}
where
%
\begin{equation}
\label{mart_brac} p(s,t):=\sqrt{\frac{2-2H}{\lambda_H}}s^{\sfrac{1}2-H}(\Psi g) (s,t)
\end{equation}
solves the equation 
%
\begin{eqnarray}\label{M_brak_smallH}
p(s,t) + \int_0^t p(r,t)
\widetilde K_H(r,s) \,dr=\sqrt{\frac
{2-2H}{\lambda_H}}s^{\sfrac{1}2-H},
\nonumber\\[-9pt]\\[-9pt]
\eqntext{0\le s\le t\le T.}
\end{eqnarray}
Moreover, $p^2(t,t)>0$ for all $t>0$ and $\F^M_t=\F^{\widetilde X}_t$.
\end{lem}

\begin{pf}
Let us set $C:= \sqrt{\frac{\lambda_H}{2-2H}}$ for brevity.
The equation \eqref{M_brak_smallH} is obtained from equation \eqref
{f_i_e} by replacing $g$ in the first term with [see \eqref{invK}]
\begin{eqnarray*}
g(s,t) &=& - c_H s^{1/2-H} \frac{d}{ds} \int
_s^t (\Psi g) (r,t) (r-s)^{1/2-H} \,dr
\\
&=&
- c_H C s^{1/2-H} \frac{d}{ds} \int
_s^t p(r,t) r^{H-1/2}
(r-s)^{1/2-H} \,dr.
\end{eqnarray*}
Indeed,
\begin{eqnarray*}
&& c_H (\Phi g ) (s)
\\
&&\qquad =
-c_H^2 C \frac{d}{ds}\int_0^s
r^{1-2H} (s-r)^{1/2-H} \frac
{d}{dr}
\\
&&\quad\qquad{}\times \int
_r^t p(u,t) u^{H-1/2}
(u-r)^{1/2-H} \,du\,dr
\\
&&\qquad =
s^{1/2-H} C \int
_0^t p(u,t) \beta_H(su)^{H-1/2}
\\
&&\quad\qquad{}\times
\int_0^{s\wedge u} r^{1-2H}
(s-r)^{-\sfrac{1} 2-H} (u-r)^{-\sfrac{1} 2-H}\,dr \,du
\\
&&\qquad =
s^{1/2-H} C \int_0^t p(u,t)
\widetilde K_H(u,s) \,du,
\end{eqnarray*}
where we used the definition \eqref{tilde_kappa} of $\widetilde K_H$.
Equation \eqref{M_brak_smallH} follows, since
\[
c_H(\Phi1) (s) = \frac{2-2H}{\lambda_H} s^{1-2H}. %
\]

Integrating by parts, we get the first formula in \eqref{MtildeX}:
\begin{eqnarray*}
&& \int_0^t p(s,t)\,d\widetilde X_s
\\
&&\qquad = p(t,t)\widetilde X_t-\int_0^t
\widetilde X_s p'(s,t)\,ds
\\
&&\qquad =
p(t,t)\widetilde X_t-\int_0^t
\int_0^s \tilde\rho(r,s) \,dX_r
p'(s,t)\,ds
\\
&&\qquad =
\int_0^t p(t,t)\tilde\rho(r,s)\,dX_r-\int_0^t \int
_r^t \tilde\rho(r,s) p'(s,t)\,ds
\,dX_r
\\
&&\qquad =
\int_0^t \int_r^t
\tilde\rho'(r,s) p(s,t)\,ds \,dX_r= \int
_0^t g(s,t) \,dX_s,
\end{eqnarray*}
where the last equality holds by a direct calculation, using the
definitions \eqref{def_rho} and \eqref{mart_brac}.

The second formula is obtained, using the identity \eqref{inner_prod}:
\begin{eqnarray*}
\langle M\rangle_t &=& \int_0^t
g(s,t)\,ds= \frac{2-2H}{\lambda_H}\int_0^t
s^{1-2H}(\Psi g) (s,t)\,ds
\\
&=&
\sqrt{\frac{2-2H}{\lambda_H}}\int_0^tp(s,t)s^{\sfrac{1}2-H}\,ds=
\int_0^t p^2(s,s)\,ds,
\end{eqnarray*}
where the last equality holds by Lemma~\ref{lem-qdiag} and
$p^2(t,t)>0$ for all $t>0$ by Corollary~\ref{corf}.


It is left to verify the inclusion $\F^{\widetilde X}_t\subseteq\F
^M_t$, or equivalently, $\E(\widetilde X_t\mid\F^M_t
)=\widetilde X_t$.
Since $\widetilde X_t-\E(\widetilde X_t\mid\F^M_t )$ is
measurable with respect to $\F^{\widetilde X}_t$ and $\widetilde
K_H(s,t)$ is weakly singular,
it admits the representation (see Remark~\ref{rem-Xrep})
\[
\widetilde X_t-\E\bigl(\widetilde X_t\mid
\F^M_t \bigr) = \int_0^t
h(r,t)\,d\widetilde X_r %
\]
with $h(\cdot,t)\in L^2([0,t])$. By the orthogonality property of the
conditional expectation,
\[
\E M_s \int_0^t h(r,t)\,d
\widetilde X_r =0, \qquad s\le t. %
\]
Let us show that this condition implies that $h(s,t)=0$ for all $s\le
t$, therefore, completing the proof.
Indeed, for $s\le t$,
\begin{eqnarray*}
&& \E M_s \int_0^t h(r,t)\,d
\widetilde X_r
\\
&&\qquad = \E\int_0^s p(u,s)\,d
\widetilde X_u \int_0^t h(r,t)\,d
\widetilde X_r
\\
&&\qquad =
\int_0^s p(r,s)h(r,t) \,dr + \int
_0^t h(r,t) \int_0^s
p(u,s)\widetilde K(r,u) \,du \,dr
\\
&&\qquad =
\int_0^s p(r,s)h(r,t) \,dr + \int
_0^t h(r,t) \bigl(C r^{\sfrac{1}2-H}-p(r,s)
\bigr) \,dr
\\
&&\qquad =
C\int_0^t h(r,t) r^{\sfrac{1}2-H}\,dr-\int
_s^t h(r,t)p(r,s) \,dr.
\end{eqnarray*}
Taking the derivative with respect to $s$, we get
\[
h(s,t)p(s,s) -\int_s^t h(r,t)\dot p(r,s) \,dr=0
\]
and, since $p(s,s)>0$, it follows that $h(s,t)$ solves the Volterra equation
%
\begin{equation}
\label{Volth} h(s,t) -\int_s^t h(r,t)
\widetilde L(r,s) \,dr=0,
\end{equation}
where $\widetilde{L}(s,t)= \dot p(r,s)/p(s,s)$ solves the equation
[cf. \eqref{Leq}]
\[
\widetilde{L}(s,t) + \int_0^t
\widetilde{L}(r,t) \widetilde K_H(r,s) \,dr= -\widetilde K_H(r,s), \qquad
0\le s\le t\le T. %
\]
By Lemma~\ref{lem-sing},
$\llvert\widetilde{L}(s,t)\rrvert\le C_1 \llvert s-t\rrvert^{-2H}$
for some constant $C_1$. By
a sufficient number of iterations, \eqref{Volth} becomes
the Volterra equation with a continuous kernel, which has a unique
solution $h(s,t)\equiv0$.
\end{pf}

%
%
%
%

\subsection{Proof of part \textup{(ii)} for \texorpdfstring{$H<\frac12$}{H<1/2}}

Equation \eqref{WH_smallH} with $s:=t$ yields $g(t,t) = 0$, $t\ge0$,
since $K^H(u,1)\equiv0$.
Hence, the formula \eqref{Mqv} holds by Lemma~\ref{mart-brak}. The
calculations in \eqref{gen} are valid for any $H$, and hence
\[
\E\bigl(X_t\mid\F^M_t\bigr) = \int
_0^t \hat g(s,t)\,dM_s, %
\]
where $\hat g(s,t)$ is given by \eqref{tildeg}.
Finally, $\F^M_t=\F^{\widetilde X}_t=\F^X_t$, by Lemmas~\ref
{mart-brak} and~\ref{lem-tildeX}. 

\section{Proof of Theorem \texorpdfstring{\protect\ref{teo3}}{2.7}}

As discussed in the \hyperref[sec1]{Introduction}, some of the assertions in this
theorem have been previously proved by a number of authors, using
different methods.
%
%
Our objective is to show how all these results can be deduced from the
canonical representation of Theorem~\ref{teo1}.
The original contribution here is the new density formula \eqref{dmuBh}.

%
%
%

\subsection{Proof of \textup{(i)}}
The fBm $B^H$ and hence also $X$ have infinite quadratic variation for
$H\in(0,\frac{1} 2 )$. Hence, $X$ is not a semimartingale
in its own filtration and a fortiori $\mu^X$ and $\mu^W$ are
singular. For $H=\frac{1} 2$, the statement of the theorem is evident.
Below we focus on the case $H\in(\frac{1} 2,1 ]$.

\begin{rem}
The fact that $X$ is not a semimartingale for $H\in(\frac{1} 2, \frac{3}4
]$ implies singularity of $\mu^X$ and $\mu^W$, but not vise versa.
For the sake of completeness, we prove both assertions directly,
showing how they stem from the same property of the kernel $K_H$ in
\eqref{K}.
\end{rem}

\subsubsection{Equivalence for \texorpdfstring{$H\in(\frac{3}4,1]$}{Hin(3/4,1]}}\label{sec-e}

Recall that for $H>\frac{1} 2$ the second term in~\eqref{Mqv} vanishes and
\[
\langle M\rangle_t = \int_0^t
g^2(s,s)\,ds, \qquad t\in[0,T]. %
\]
By Lemma~\ref{cor}, $g(t,t)>0$ for all $t\ge0$ and the process
%
\begin{equation}
\label{WW} \overline{B}_t =\int_0^t
\frac{1} {g(s,s)} \,dM_s
\end{equation}
is a Brownian motion in filtration $\F^X_t$.
On the other hand,
\begin{eqnarray*}
M_t &=& \int_0^t
g(s,t)\,dX_s = \int_0^t
g(s,s)\,dX_s + \int_0^t
\bigl(g(r,t)-g(r,r) \bigr)\,dX_r
\\
&=&
\int_0^t g(s,s)\,dX_s + \int
_0^t \int_r^t
\dot g(r,s)\,ds\,dX_r
\\
&=& \int_0^t
g(s,s)\,dX_s +\int_0^t \int
_0^s \dot g(r,s)\,dX_r \,ds,
\end{eqnarray*}
where the last equality holds since $\dot g(\cdot,s)\in L^2([0,s])$ by
Lemma~\ref{lem-dif}.
Hence,
\[
\overline{B}_t = \int_0^t
\frac{1}{g(s,s)} \,dM_s= X_t +\int_0^t
\int_0^s \frac{\dot g(r,s)}{g(s,s)}\,dX_r
\,ds =: X_t +\int_0^t
\varphi_s(X)\,ds. %
\]
The desired claim follows by the Girsanov theorem (Theorem 7.7 in \cite
{LS1}), once we check
%
\begin{equation}
\label{verme} \int_0^T \E
\varphi_t^2(\overline{B}) \,dt<\infty\quad\mbox{and}\quad
\int_0^T \E\varphi_t^2(X)
\,dt<\infty.
\end{equation}
Since\vspace*{1pt} $\varphi_t(\cdot)$ is a linear functional of $X=B+B^H$ and $B$
and $B^H$ are independent, it is enough to check
only the latter condition. The function $L(s,t)= \frac{ \dot g(s,t)}{g(t,t)}$
satisfies~\eqref{Leq}, and hence for $H>3/4$
\begin{eqnarray*}
\E\varphi_t^2(X) &=& \E\biggl(\int_0^tL(r,t)
\,dX_r \biggr)^2
\\
&=&
\int_0^tL^2(s,t) \,ds + \int
_0^t\int_0^t
L(s,t)L(r,t)K_H(r,s)\,dr\,ds
\\
&=&
\int_0^t L(s,t) \biggl(L(s,t) + \int
_0^t L(r,t)K_H(r,s)\,dr \biggr) \,ds
\\
&=&
- \int_0^t L(s,t) K_H(s,t) \,ds
\\
&\le&
\biggl(\int_0^t L^2 (s,t)\,ds
\biggr)^{1/2} \biggl(\int_0^t
K_H^2(s,t) \,ds \biggr)^{1/2}
\\
&=& C_1
\biggl(\int_0^t L^2 (s,t)\,ds
\biggr)^{1/2} t^{2H-3/2}.
\end{eqnarray*}
Since the kernel is positive definite, multiplying \eqref{Leq} by
$L(s,t)$ and integrating gives
\[
\int_0^t L^2 (s,t)\,ds \le- \int
_0^t L(s,t) K_H(s,t) \,ds \le
C_1 \biggl(\int_0^t
L^2 (s,t)\,ds \biggr)^{1/2} t^{2H-3/2}, %
\]
and consequently
\[
\biggl(\int_0^t L^2 (s,t)\,ds
\biggr)^{1/2} \le C_1 t^{2H-3/2}. %
\]
Plugging this bound back gives
$
\E\varphi_t^2(X)\le
C_1^2 t^{4H-3}
$ and in turn
\[
\int_0^T \E\varphi_t^2(X)
\,dt\le C_1^2 \int_0^T
t^{4H-3} \,dt = C_2 T^{4H-2}, %
\]
which verifies \eqref{verme} and completes the proof.

\subsubsection{Singularity for \texorpdfstring{$H\in(\frac{1}2,\frac{3}4]$}{Hin(1/2,3/4]}}

Suppose there exists a probability measure~$\Q$, equivalent to $\P$,
under which $X$ is a Brownian motion
in its natural filtration. Since the semimartingale property is
preserved under equivalent change of
measure, the $\P$-martingale
\[
M_t = \int_0^t
g(s,t)\,dX_s, \qquad t\in[0,T], %
\]
must be a semimartingale under $\Q$. Since $X$ is assumed to be a
Brownian motion under $\Q$, this is equivalent to
saying that the process
\[
L_t:=\int_0^t g(s,t)\,d
\overline{B}_s, %
\]
with the Brownian motion $\overline{B}$, defined in \eqref{WW}, must
be a semimartingale under $\P$.

We will argue by contradiction that this is impossible for $H \le\frac
{3} 4$.
To this end, we define
\[
\psi(s,t) = -\int_s^t g(r,r) \sum
_{m=1}^{n_0-1}(-1)^m K_H^{(m)}(r,s)\,dr,
\qquad0< s< t\le T, %
\]
where $n_0$ is the least integer greater than $\frac{1}{4H-2}$.
Note that $\psi(\cdot,t)\in L^2([0,t])$ and define the processes
\begin{eqnarray*}
U_t &:=& \int_0^t \psi(s,t)\,d
\overline{B}_s,
\\
V_t &:=& \int_0^t
\bigl(g(s,t)-g(s,s) + \psi(s,t) \bigr)\,d\overline{B}_s,
\end{eqnarray*}
so that
\[
L_t = V_t + \int_0^t
g(s,s)\,d\overline{B}_s -U_t. %
\]
The second term is a martingale in filtration $\F^X_t$ and hence, in
order to argue that $L$ is not a semimartingale, it is enough to show that:
\begin{longlist}[(a)]
\item[(a)] $U$ has zero quadratic variation, but infinite variation,

\item[(b)] $V$ has bounded variation.\vadjust{\goodbreak}
\end{longlist}

\begin{pf*}{Proof of \textup{(a)}} To verify this assertion, we will need an estimate
for the variance of increments of $U$. To this end,
for any two points $t_1,t_2\in[0,T]$, such that $0<t_2-t_1<1$,
%
\begin{eqnarray}\label{dU}
\E(U_{t_2}-U_{t_1}
)^2 &=&
\E\biggl( \int_{t_1}^{t_2} \psi(s,t_2)\,d
\overline{B}_s +\int_0^{t_1} \bigl(
\psi(s,t_2)-\psi(s,t_1) \bigr)\,d\overline{B}_s
\biggr)^2
\nonumber\\[-8pt]\\[-8pt]\nonumber
&=& \int_{t_1}^{t_2} \psi^2(s,t_2)\,ds
+\int_0^{t_1} \bigl(\psi(s,t_2)-
\psi(s,t_1) \bigr)^2\,ds.
\end{eqnarray}
To bound the first term, note that
\begin{eqnarray*}
\psi^2(s,t_2) &\le&\llVert g\rrVert^2_\infty
n_0 \sum_{m=1}^{n_0-1} \biggl(
\int_s^{t_2} K_H^{(m)}(s,r)\,dr
\biggr)^2
\\
&\le&
C_1 \sum_{m=1}^{n_0-1}
(t_2-s)^{(4H-2)m} \le C_2 (t_2-s)^{4H-2},
\end{eqnarray*}
where $\llVert g\rrVert_\infty=\sup_{r\le T}\llvert
g(r,r)\rrvert<\infty$,
and consequently
\[
\int_{t_1}^{t_2} \psi^2(s,t_2)\,ds
\le C_3 (t_2-t_1)^{4H-1}.
\]
For the second term, we have
%
\begin{eqnarray}\label{manyterms}
&& \int_0^{t_1}
\bigl(\psi(s,t_2)-\psi(s,t_1) \bigr)^2\,ds\nonumber
\\
&&\qquad =
\int_0^{t_1} \Biggl(\sum
_{m=1}^{n_0-1}\int_{t_1}^{t_2}
(-1)^m g(r,r) K_H^{(m)}(s,r)\,dr
\Biggr)^2\,ds
\nonumber\\[-8pt]\\[-8pt]\nonumber
&&\qquad =
\sum_{m=1}^{n_0-1}\sum
_{\ell=1}^{n_0-1} \int_0^{t_1}
\int_{t_1}^{t_2} \int_{t_1}^{t_2}(-1)^{m+\ell}
g(r,r)g(\tau,\tau)
\\
&&\quad\qquad{}\times  K_H^{(m)}(s,r)K_H^{(\ell)}(s,
\tau)\,dr\,d\tau \,ds.\nonumber
\end{eqnarray}
The dominating term in this sum corresponds to $m=1$, $\ell=1$:
\[
\int_0^{t_1} \biggl( \int_{t_1}^{t_2}
g(r,r) K_H(r,s) \,dr \biggr)^2\,ds. %
\]
We have
%
\begin{eqnarray}\label{domterm}
&& \int_0^{t_1}\biggl( \int_{t_1}^{t_2} K_H(r,s)\,dr
\biggr)^2\,ds \nonumber
\\
&&\qquad =
H^2\int_0^{t_1} \bigl(
(t_2-t_1+s)^{2H-1}-s^{2H-1}
\bigr)^2\,ds
\\
&&\qquad =
H^2(t_2-t_1)^{4H-1}\int
_0^{\afrac{t_1}{t_2-t_1}} \bigl( (1+u)^{2H-1}-u^{2H-1}
\bigr)^2\,du.\nonumber
\end{eqnarray}
The increasing function
\[
\gamma(y):= H^2 \int_0^{y} \bigl(
(1+u)^{2H-1}-u^{2H-1} \bigr)^2\,du, \qquad y\ge0
\]
satisfies
\begin{eqnarray*}
\lim_{y\to\infty}\gamma(y) &=& \gamma_H,\qquad  H\in
\bigl(\tfrac{1}2 , \tfrac{3} 4 \bigr),
\\
\lim_{y\to\infty}\frac{\gamma(y)}{\log y} &=& \gamma_{3/4} ,\qquad
 H=\tfrac{3} 4,
\end{eqnarray*}
with positive constants $\gamma_H$. The function $r\mapsto g(r,r)$ is
positive and continuous on $[0,T]$, and hence
\[
c_4\le\int_0^{t_1} \biggl( \int
_{t_1}^{t_2} g(r,r) K_H (s,r) \,dr
\biggr)^2 \Big/ (t_2-t_1)^{4H-1}\gamma
\biggl(\frac{t_1}{t_2-t_1} \biggr) \le C_4 %
\]
with some positive constants $c_4$, $C_4$ for all sufficiently small $t_2-t_1$.
A similar calculation shows that the rest of the terms in \eqref
{manyterms} converge to zero as $t_2-t_1\to0$ at a faster rate
and assembling all parts together, we obtain
%
\begin{equation}
\label{cC} c_5\le\E(U_{t_2}-U_{t_1}
)^2 /(t_2-t_1)^{4H-1}\gamma\biggl(
\frac{t_1}{t_2-t_1} \biggr)\le C_5.
\end{equation}

Now let $0=t_0<t_1<\cdots<t_n=T$ be an arbitrary partition, then for
all $H\in(\frac{1} 2, \frac{3} 4]$
\begin{eqnarray*}
\E\sum_{i=1}^n (U_{t_i}-U_{t_{i-1}}
)^2 &\le& C_5 \sum_{i=1}^n
(t_i-t_{i-1})^{4H-1}\gamma\biggl(
\frac{T}{t_i-t_{i-1}} \biggr)
\\
&\le&
C_6 \max_{i}(t_i-t_{i-1})^{4H-2}
\log\frac{1}{t_i-t_{i-1}}\xrightarrow{n\to\infty} 0,
\end{eqnarray*}
that is, $U$ has zero quadratic variation.

On the other hand, since the process $U$ is Gaussian
\begin{eqnarray*}
\E\sum_{i=1}^n \llvert
U_{t_i}-U_{t_{i-1}}\rrvert&\ge&\sqrt{\frac{2}\pi}c_5 \sum_{i:t_i\ge T/2} (t_i-t_{i-1})^{2H-\sfrac{1}2}
\gamma^{\sfrac{1}2} \biggl(\frac{T/2}{t_i-t_{i-1}} \biggr)
\\
&\ge&
c_6 \min_i (t_i-t_{i-1})^{2H-\sfrac{3} 2}
\gamma^{1/2} \biggl(\frac
{T/2}{t_i-t_{i-1}} \biggr) \xrightarrow{n\to\infty}
\infty,
\end{eqnarray*}
which implies that $U$ has infinite variation (see, e.g., Theorem 4,
Chapter~4, Section~9 in \cite{LS89}).
\end{pf*}

\begin{pf*}{Proof of \textup{(b)}}
For $0< s< t\le T$
\[
\dot\psi(s,t) := \frac{\partial}{\partial t}\psi(s,t) = - g(t,t) \sum
_{m=1}^{n_0-1}(-1)^mK_H^{(m)}(s,t)
\]
and hence
\begin{eqnarray*}
&& \int_0^t \dot\psi(s,t)
K_H(r,s) \,dr
\\
&&\qquad = -\int_0^t \Biggl(
g(t,t) \sum_{m=1}^{n_0-1}(-1)^mK_H^{(m)}(s,t)
\Biggr)K_H(r,s) \,dr
\\
&&\qquad =
-g(t,t) \sum_{m=1}^{n_0-1}(-1)^mK_H^{(m+1)}(s,t)
= g(t,t) \sum_{m=2}^{n_0}(-1)^{m}K_H^{(m)}(s,t)
\\
&&\qquad =
g(t,t) K_H(s,t) -\dot\psi(s,t)+(-1)^{n_0}g(t,t)
K_H^{(n_0)}(s,t).
\end{eqnarray*}
Adding this expression to the equation for $\dot g(s,t)$ [see \eqref
{WHdot}], we get
\begin{eqnarray*}
&& \bigl(\dot g(s,t)+\dot\psi(s,t) \bigr) + \int_0^t
\bigl(\dot g(r,t)+\dot\psi(r,t) \bigr) K_H(r,s)\,dr
\\
&&\qquad =
(-1)^{n_0} g(t,t) K_H^{(n_0)}(s,t).
\end{eqnarray*}
By the choice of $n_0$, the right-hand side is square integrable and so
is the function $\dot g(s,t)+\dot\psi(s,t)$, $s\in(0,t)$.
Since $\psi(s,s)=0$,
\begin{eqnarray*}
V_t &=& \int_0^t
\bigl(g(s,t)-g(s,s) + \psi(s,t) \bigr)\,d\overline{B}_s
\\
&=& \int
_0^t\int_s^t
\bigl( \dot g(s,r) + \dot\psi(s,r) \bigr)\,dr\,d\overline{B}_s
\\
&=&
\int_0^t\int_0^r
\bigl( \dot g(s,r) + \dot\psi(s,r) \bigr)\,d\overline{B}_s\,dr,
\end{eqnarray*}
and hence $V$ has bounded variation. 
\end{pf*}

\subsubsection{$X$ is not a semimartingale for \texorpdfstring{$H\in(\frac{1}2, \frac{3}4]$}{Hin(1/2,3/4]}}\label{sec:notsem}

We will use the representation \eqref{tildeg}, where, for $H>\frac{1} 2$,
\[
\hat g(s,t) = 1-\frac{1}{g(s,s)}\int_0^t
L(r,s)\,dr. %
\]
We have
\begin{eqnarray*}
X_t &=& M_t - \int_0^t
\frac{1}{g(s,s)} \int_0^t L(\tau,s) \,d\tau
\,dM_s
\\
&=& M_t - \int_0^t
\int_0^t L(\tau,s) \,d\tau \,d
\overline{B}_s
\\
&=&
M_t - \int_0^t \int
_0^s L(\tau,s) \,d\tau \,d\overline{B}_s
- \int_0^t \int_s^t
L(\tau,s) \,d\tau \,d\overline{B}_s =: M_t - N_t
- U_t,
\end{eqnarray*}
where $\overline{B}$ is the Brownian motion, defined by \eqref{WW}.
By Lemma~\ref{lem-sing}, the function $\int_0^s L(\tau,s) \,d\tau$ is
bounded, and hence $M-N$ is a martingale in filtration $\F^X_t$.
To argue that $X$ is not a semimartingale in its own filtration, we
will show that
$U$ has zero quadratic variation, but infinite variation.

Let $n_0$ be the least integer greater than $\frac{1} {4H-2}$.
It then follows from \eqref{Leq} that the function
\[
Q(s,t) := \int_0^t L(r,t)K_H^{(n_0-1)}(r,s)\,dr
\]
satisfies
\[
Q(s,t) + \int_0^t Q(r,t)
K_H(r,s)\,dr = - K_H^{(n_0)}(s,t), %
\]
and hence $Q(\cdot,t)\in L^2([0,t])$. Iterating the equation \eqref
{Leq}, we get
%
\begin{equation}
\label{RQrep} L(s,t) = \sum_{m=1}^{n_0-1}
(-1)^m K_H^{(m)}(s,t) + (-1)^{(n_0-1)}
Q(s,t).
\end{equation}
Define $\phi(s,t):= \int_s^t L(\tau,s) \,d\tau$, then, similar to
\eqref{dU},
%
\begin{equation}
\label{UU} \E(U_{t_2}-U_{t_1} )^2 = \int
_{t_1}^{t_2} \phi^2(s,t_2)\,ds
+\int_0^{t_1} \bigl(\phi(s,t_2)-
\phi(s,t_1) \bigr)^2\,ds.
\end{equation}
By \eqref{RQrep},
\begin{eqnarray*}
\phi^2(s,t) &\le& C_1 \sum_{m=1}^{n_0-1}
\biggl(\int_s^t K_H^{(m)}(
\tau,s) \,d\tau\biggr)^2 +
C_1 \biggl(\int_s^t Q(\tau,s) \,d
\tau\biggr)^2
\\
&\le& C_2 \llvert t-s\rrvert^{4H-2}
\end{eqnarray*}
and hence the first term in \eqref{UU} is bounded by
\[
\int_{t_1}^{t_2} \phi^2(s,t_2)\,ds
\le\int_{t_1}^{t_2} C_2
(t_2-s)^{4H-2}\,ds \le C_3 (t_2-t_1)^{4H-1}.
\]
Further,
\begin{eqnarray*}
\int_0^{t_1} \bigl(\phi(s,t_2)-
\phi(s,t_1) \bigr)^2\,ds &=& \int_0^{t_1}
\biggl(\int_{t_1}^{t_2} L(\tau,s) \,d\tau
\biggr)^2\,ds
\\
&=&
\int_0^{t_1} \int_{t_1}^{t_2}
\int_{t_1}^{t_2} L(\tau,s)L(r,s) \,d\tau \,dr \,ds.
\end{eqnarray*}
After plugging in the expression \eqref{RQrep}, the dominating term is
readily seen to be given by \eqref{domterm},
and hence as in the previous section the bound \eqref{cC} holds.
Consequently, $U$ has infinite variation and zero quadratic variation and
thus $X$ is not a semimartingale.

\subsection{Proof of \textup{(ii)}}

\subsubsection{Equivalence for \texorpdfstring{$H<\frac{1}4$}{H<1/4}}
By Lemma~\ref{mart-brak},
\begin{eqnarray*}
M_t &= & \int_0^t p(s,t) \,d
\widetilde X_s = \int_0^t p(s,s) \,d
\widetilde X_s + \int_0^t \int
_s^t \dot p(s,r)\,dr \,d\widetilde X_s
\\
&=&
 \int_0^t p(s,s) \,d\widetilde X_s + \int_0^t \int
_0^r \dot p(s,r) \,d\widetilde X_s
\,dr.
\end{eqnarray*}
The last equality holds since $\dot p(\cdot,t)\in L^2([0,t])$ for
$H<\frac{1} 4$ by
Lemma~\ref{lem-dif}. Hence,
\[
\widetilde X_t = \overline{W}_t - \int
_0^t \tilde\varphi_s(
\widetilde X) \,ds,
\]
where
$
\overline{W}_t = \int_0^t \frac{dM_s}{p(s,s)}
$
is a Brownian motion in filtration $\F^{\widetilde X}_t$ and
\begin{eqnarray*}
\tilde\varphi_t(\widetilde X)&= & \int_0^t
\widetilde L(s,t) \,d\widetilde X_s= \int_0^t
\sqrt{\frac{2-2H}{\lambda_H}} \bigl(\Psi^{-1}u^{H-1/2}\widetilde L(u,t) \bigr) (s,t) \,d X_s
\\
&=:&
\int_0^t L(s,t) \,d X_s =:
\varphi_t(X),
\end{eqnarray*}
with
$
\widetilde L(s,t):=\frac{\dot p(s,t)}{p(t,t)}$.
A calculation shows that
%
\begin{equation}
\label{subsme} L(s,t) = \frac{\dot g(s,t)}{p(t,t)}-\dot{\tilde\rho}(s,t).
\end{equation}
Since the kernel in \eqref{kappabar} is weakly singular, by Lemma~\ref
{lem-dif}, the solution $g(s,t)$ of \eqref{WH_smallH}
is differentiable with respect to the second variable. Taking the
derivative of~\eqref{f_i_e}, we obtain
\[
c_H (\Phi\dot g) (s) + \frac{2-2H}{\lambda_H} (\Psi\dot g) (s,t)
s^{1-2H} =0, \qquad0<s< t\le T, %
\]
since $g(t,t)=0$ for $H<\frac{1} 2$. Multiplying this equation by $\frac
{\lambda_H}{2-2H}s^{2H-1}$
and applying $\Psi^{-1}$, it can be seen that $\dot{g}(s,t)$
satisfies [cf. \eqref{WHdot}]
%
\begin{equation}
\label{dotgeq} \dot{g}(s,t) + \beta_Ht^{-2H}\int
_0^t\dot{ g}(r,t)K^H \biggl(
\frac
{r}{t},\frac{s}{t} \biggr) \,dr=p(t,t)\dot{\tilde{\rho}}(s,t),
\end{equation}
and plugging \eqref{subsme} into this equation yields
\[
L(s,t) + \beta_Ht^{-2H}\int_0^t
\bigl(L(r,t) + \dot{\tilde\rho}(r,t) \bigr) K^H \biggl(
\frac{r}{t},\frac{s}{t} \biggr) \,dr=0. %
\]
After applying $\Psi$ and rearranging the terms, it becomes
\[
L(s,t) + \frac{\partial}{\partial s}\int_0^t L(r,t)
\frac{\partial
}{\partial r} R(r,s) \,dr = - \dot{\tilde\rho}(s,t). %
\]
Assembling all the parts together, we get
\[
\widetilde X_t = \overline{W}_t - \int
_0^t \varphi_s( X) \,ds,
\]
and, consequently, the representation \eqref{BHt}:
\[
X_t = \int_0^t \rho(s,t) \,d
\overline{W}_t - \int_0^t
\rho(s,t) \varphi_s( X) \,ds=: \overline{B}^H_t
- \int_0^t \rho(s,t) \varphi_s(
X) \,ds. %
\]
The density \eqref{dmuBh} is obtained by Girsanov's change of measure
as in Section~\ref{sec-e}, under which $\widetilde X$ is a Brownian motion
and, therefore, $X$ is an fBm.

\subsubsection{Singularity for \texorpdfstring{$H\ge\frac{1}4$}{H>=1/4}}
The claim is obvious for $H=\frac{1} 2$.
For $H> \frac{1} 2$, the process $X$ has positive quadratic variation,
and hence cannot be equivalent to fBm with $H> \frac{1} 2$, whose
quadratic variation vanishes.

To prove singularity for $H\in[\frac{1} 4, \frac{1} 2)$, suppose there
is a probability $\Q$, equivalent to $\P$, under which $X$ is an fBm
with the
Hurst exponent $H$ in its own filtration. Then $\widetilde X_t = \int
_0^t \tilde\rho(s,t)\,dX_s$, with $\tilde\rho(s,t)$ defined in
\eqref{rhorho}, is a Brownian motion under $\Q$. By calculations as
in Section~\ref{sec:notsem}, it can be seen that $\widetilde X$ is not
a semimartingale for $H\in[\frac{1} 4, \frac{1} 2)$, thus obtaining a
contradiction.

\section{Proof of Corollaries \texorpdfstring{\protect\ref{teo4}}{2.9} and \texorpdfstring{\protect\ref{teo4b}}{2.10}}\label{sec-thm4}

The proofs of Corollaries~\ref{teo4} and~\ref{teo4b} follow the same
pattern and we will omit the details for the latter.
The representation \eqref{reprZ} is obvious in view of \eqref{MX} and
the definition
\eqref{Kr2}. To prove the inversion formula \eqref{repYG}, we have to
check that
%
\begin{equation}
\label{G_F} \int_0^t \xi_s\,ds=
\int_{0}^{t}\hat g(s,t)\Xi(s) \,d\langle M
\rangle_{s},\qquad t\in[0,T].
\end{equation}
Since this is a pathwise statement and $\xi$ is the only random
object, no generality will be lost if
$\xi_t$ is assumed to be deterministic. For $\xi\in L^2([0,t])$, we have
\begin{eqnarray*}
\E\biggl(\int_{0}^{t}\xi_s
\,dB_s\Big| \F^X_t \biggr) &=& \E\biggl(
\int_{0}^{t}\xi_s \,dB_s
\Big|  \F^M_t \biggr)
\\
&=&
\int_0^t \frac{d}{d\langle M\rangle_s} \biggl(\E
M_s \int_0^t \xi_r
\,dB_r \biggr) \,dM_s
\\
& =&
\int_0^t \frac{d}{d\langle M\rangle_s} \biggl(\E\int
_0^s g(r,s)\,dX_r \int
_0^t \xi_r \,dB_r
\biggr) \,dM_s
\\
&=& \int_{0}^{t}\Xi(s) \,d
M_s,
\end{eqnarray*}
and, using the representation \eqref{XG}, we obtain \eqref{G_F}:
\[
\int_{0}^{t}\xi_s \,ds=\E
X_t \int_{0}^{t}\xi_s
\,dB_s= \E X_t \int_{0}^{t}
\Xi(s) \,d M_s= \int_{0}^{t}\hat g(s,t)
\Xi(s) \,d\langle M \rangle_s. %
\]

The formula \eqref{RNflaXY} follows from
Theorem 7.13 in \cite{LS1}, once we check
%
\begin{equation}
\label{checkme1} \int_0^T
\Xi^{2}(t) \,d\langle M\rangle_{t} <\infty, \qquad\P\mbox{-a.s.}
\end{equation}
and
%
\begin{equation}
\label{checkme2} \E\int_0^T \bigl\llvert
\Xi(t)\bigr\rrvert \,d\langle M\rangle_{t} <\infty.
\end{equation}
Let us first consider the case $H>\frac{1} 2$, for which $d\langle M
\rangle_t/dt=g^2(t,t)>0$.
By definition \eqref{Kr2} and continuity of $\xi_t$
\[
\Xi(t)g(t,t)= \xi_t + \int_0^{t}L(s,t)
\xi_s \,ds, %
\]
where $L(s,t)$ solves \eqref{Leq}. By (ii) of Lemma~\ref{lem-dif},
$\llvert L(s,\tau)\rrvert\le c_1\llvert s-\tau\rrvert
^{2H-2}$ with a constant $c_1$ and,
therefore,
\begin{eqnarray*}
\biggl\llvert\int_0^{\tau}L(s,\tau)
\xi_s \,ds\biggr\rrvert&\le&\biggl(\int_0^{\tau}
\bigl\llvert L(s,\tau)\bigr\rrvert\xi_s^2 \,ds
\biggr)^{1/2} \biggl(\int_0^{\tau}\bigl
\llvert L(s,\tau)\bigr\rrvert \,ds \biggr)^{1/2}
\\
&\le&
c_2 \biggl(\int_0^{T}\bigl
\llvert L(s,\tau)\bigr\rrvert\xi^2_s \,ds
\biggr)^{1/2},
\end{eqnarray*}
where $ c^2_2=c_1 \sup_{\tau\in[0,T]}\int_0^{T}\llvert s-\tau
\rrvert^{2H-2} \,ds$. Consequently,
\begin{eqnarray*}
\int_0^T \Xi^{2}(t)\,d\langle M
\rangle_t &\le&2 \int_0^T \xi
^2_t\,dt+ 2\int_0^T
\biggl(\int_0^{t}L(s,t)\xi_s \,ds
\biggr)^2 \,dt
\\
&\le&
2 \int_0^T \xi^2_s
\,ds+ 2c_2^2 \int_0^T
\xi^2_s \int_0^{T}
\bigl\llvert L(s,t)\bigr\rrvert \,dt \,ds
\\
&\le& 2\bigl(1 + c_2^4
\bigr) \int_0^T \xi_t^2
\,dt <\infty,
\end{eqnarray*}
which proves \eqref{checkme1}.
Condition \eqref{checkme2} is verified similarly:
\begin{eqnarray*}
\E\int_0^T \bigl\llvert\Xi(t)\bigr\rrvert d
\langle M\rangle_{t} &\le& c_3 \E\int_0^T
\llvert\xi_t\rrvert \,dt + c_3 \E\int
_0^T\llvert\xi_s\rrvert\int
_0^{T}\bigl\llvert L(s,t)\bigr\rrvert \,ds \,dt
\\
&\le&
c_3 \bigl(1+c_2^2 \bigr)\E\int
_0^T \llvert\xi_t\rrvert \,dt <
\infty,
\end{eqnarray*}
where $c_3:=\sup_{t\in[0,T]}g(t,t)$.

For $H<\frac{1} 2$, by Lemma~\ref{mart-brak}, $d\langle M\rangle
_t/dt=p^2(t,t)>0$
and, therefore,
\[
\Xi(t)p(t,t)= \xi_t + \int_0^{t}
\frac{\dot g(s,t)}{p(t,t)}\xi_s \,ds. %
\]
Dividing both sides of equation \eqref{dotgeq} by $p(t,t)$, we see
that $H(s,t):=\dot g(s,t)/p(t,t)$ solves the equation
\[
H(s,t) + \beta_Ht^{-2H}\int_0^tH(r,t)K^H
\biggl(\frac{r}{t},\frac
{s}{t} \biggr) \,dr= \dot{\tilde{\rho}}(s,t), %
\]
where\vspace*{1pt} $\llvert K^H(s,t)\rrvert \le c_4 \llvert s-t\rrvert^{-2H}$
and $\llvert\dot{\tilde{\rho
}}(s,t)\rrvert\le c_5\llvert s-t\rrvert^{-\sfrac{1} 2 -H}$.
Therefore, by Lemma~\ref{lem-sing}, $\llvert H(s,t)\rrvert\le
c_6 \llvert s-t\rrvert^{-\sfrac{1} 2-H}$ and the\vspace*{2pt} claim follows by the same arguments as in the case
$H<\frac{1} 2$.

\section{The mixed Riemann--Liouville process}\label{sec-RL}

In this section, we outline the results, obtained by our method, for
the mixed Riemann--Liouville process:
\[
X_t = B_t + V^H_t, \qquad t
\in[0,T], %
\]
where $V^H$ is defined in \eqref{RLproc}.

As mentioned in the \hyperref[sec1]{Introduction}, $V^H$ shares many properties with
$B^H$. In particular, the respective
stochastic calculus builds on operators similar to those defined in
\eqref{Psi}--\eqref{invQ}.
In this case, they are defined in a slightly different way:
\begin{eqnarray*}
(\Psi f) (s,t) &=& -2H\frac{d}{ds}\int_s^t
f(r) (r-s)^{H-\sfrac{1}2} \,dr ,\qquad0\leq s\leq t ,
\\
(\Phi f) (s) &=& c_H\frac{d}{ds}\int_0^s
f(r) (s-r)^{\sfrac{1}2-H} \,dr,
\end{eqnarray*}
and
\begin{eqnarray*}
\bigl(\Psi^{-1} g\bigr) (s,t) &=& -c_H \frac{d}{ds}
\int_s^t (r-s)^{1/2-H} g(r)\,dr,
\\
\bigl(\Phi^{-1} g\bigr) (s) &=& \frac{2H} {c_H}\frac{d}{ds}
\int_0^s g(r) (s-r)^{H-1/2}\,dr.
\end{eqnarray*}
Stochastic integrals with respect to $V^H$ can be defined on the space
\[
\Lambda_t : = \biggl\{f: [0,t]\mapsto\Real\mbox{ such that }
\int_0^t (\Psi f)^2(s,t) \,ds <
\infty\biggr\}, %
\]
with the scalar product
\[
\langle f,g\rangle_{\Lambda_t} = \int_0^t
(\Psi f) (s) (\Psi g) (s)\,ds. %
\]
The formula \eqref{inner_prod} remains valid and
the kernels $\rho$ and $\tilde\rho$ become [cf. \eqref{rhorho}]
\[
\rho(s,t) = (\Psi1) (s,t), \qquad\tilde{\rho}(s,t) = \bigl(\Psi
^{-1} 1 \bigr) (s,t), %
\]
so that
\[
V^H_t = \int_0^t
\rho(s,t)\,dW_s, %
\]
where
\[
W_t = \int_0^t \tilde{\rho}(s,t)\,dV^H_t, %
\]
is a Brownian motion with $\F^W_t=\F^{V^H}_t$.
As before, we have
\[
\E\int_0^t f(u)\,dV^H_u
\int_0^t f(r)\,dV^H_r =
\langle f,g\rangle_{\Lambda_t}. %
\]
For $H>\frac{1} 2$, the covariance function of $V^H$
%
\begin{equation}
\label{covfRL} R(s,t)=\E V^H_tV^H_s
= (2H)^{2} \int_0^{s\wedge t}(t-r)^{H-1/2}(s-r)^{H-1/2}\,dr,
\end{equation}
satisfies
%
\begin{equation}
\label{KHRL} \qquad K_H(s,t):=\frac{\partial^{2}R(s,t)}{\partial t\,\partial s}
=H^2(2H-1)^2
\llvert t-s\rrvert^{2H-2}\chi\biggl(\frac
{s}{t} \biggr),\qquad s
\le t
\end{equation}
with $\chi\in C([0,1])$ given by
\[
\chi(u)=\int_{0}^{\afrac{u}{1-u}}\tau^{H-3/2}(1+
\tau)^{H-3/2} \,d\tau. %
\]

Repeating the proofs with these modifications gives the following
analogs of the main results.

\begin{teo}
\textup{(i)}~Theorem~\ref{teo1} remains valid with $g(s,t)$ solving the equation
%
\begin{eqnarray}\label{ideqRL}
g(s,t) - \frac{\partial}{\partial s}\int_0^t
R(r,s)\frac{\partial}{\partial r} g(r,t) \,dr+g(t,t) \frac{\partial
}{\partial s}R(s,t)=1,
\nonumber\\[-8pt]\\[-8pt]
\eqntext{0<s,t \le T,}
\end{eqnarray}
where $R(s,t)$ is defined in \eqref{covfRL}, and \eqref{Mqv} is
replaced with
\[
\frac{d}{dt} \langle M \rangle_t = g^2(t,t) + (
\Psi g)^2(t,t) > 0, \qquad t\in[0,T]. %
\]

\textup{(ii)}~Theorem~\ref{teo3} remains valid with $\overline{B}^H$ being
replaced with the Riemann--Liouville process $\overline{V}^H$.

\textup{(iii)}~Corollary~\ref{teo4} remains valid.
\end{teo}

Note that equation \eqref{ideqRL} can be obtained formally from
\eqref{ideq} through integration by parts. The reason for such
a twist is that the first derivative $ \partial R(s,t)/\partial s$ of
the covariance function $R(s,t)$
is not integrable for $H<\frac{1} 2$.
Let us note that \eqref{ideqRL} also reduces to a weakly singular
integral equation with the kernel $K_H$ from \eqref{KHRL} for $H>\frac
{1} 2$ and the
kernel
\[
K^H(u,v)= \llvert u-v\rrvert^{-2H}\int
_0^{\vfrac{1-u}{\llvert
u-v\rrvert}}\tau^{-1/2-H}(1+\tau
)^{-1/2-H} \,d\tau%
\]
for $H<\frac{1} 2$ (cf. Theorem~\ref{teo-eq}).


\section*{Acknowledgement} We would like to thank Alain Le Breton
for enlightening discussions and his interest in this work.


%

\printaddresses

\begin{thebibliography}{33}
\bibitem{BP88}
%
\begin{barticle}[mr]
\bauthor{\bsnm{Barton},~\bfnm{Richard~J.}\binits{R.~J.}} \AND
\bauthor{\bsnm{Poor},~\bfnm{H.~Vincent}\binits{H.~V.}}
(\byear{1988}).
\btitle{Signal detection in fractional {G}aussian noise}.
\bjournal{IEEE Trans. Inform. Theory}
\bvolume{34}
\bpages{943--959}.
\bid{doi={10.1109/18.21218}, issn={0018-9448}, mr={0982805}}
\end{barticle}
%

\bptok{imsref}%
\endbibitem

\bibitem{BN03}
%
\begin{barticle}[mr]
\bauthor{\bsnm{Baudoin},~\bfnm{Fabrice}\binits{F.}} \AND
\bauthor{\bsnm{Nualart},~\bfnm{David}\binits{D.}}
(\byear{2003}).
\btitle{Equivalence of {V}olterra processes}.
\bjournal{Stochastic Process. Appl.}
\bvolume{107}
\bpages{327--350}.
\bid{doi={10.1016/S0304-4149(03)00088-7}, issn={0304-4149}, mr={1999794}}
\end{barticle}
%

\bptok{imsref}%
\endbibitem

\bibitem{BSV07}
%
\begin{bincollection}[mr]
\bauthor{\bsnm{Bender},~\bfnm{Christian}\binits{C.}},
\bauthor{\bsnm{Sottinen},~\bfnm{Tommi}\binits{T.}} \AND
\bauthor{\bsnm{Valkeila},~\bfnm{Esko}\binits{E.}}
(\byear{2011}).
\btitle{Fractional processes as models in stochastic finance}.
In \bbooktitle{Advanced Mathematical Methods for Finance}
\bpages{75--103}.
\bpublisher{Springer},
\blocation{Heidelberg}.
\bid{doi={10.1007/978-3-642-18412-3_3}, mr={2792076}}
\end{bincollection}
%

\bptok{imsref}%
\endbibitem

\bibitem{BGT04}
%
\begin{barticle}[mr]
\bauthor{\bsnm{Bojdecki},~\bfnm{Tomasz}\binits{T.}},
\bauthor{\bsnm{Gorostiza},~\bfnm{Luis~G.}\binits{L.~G.}} \AND
\bauthor{\bsnm{Talarczyk},~\bfnm{Anna}\binits{A.}}
(\byear{2004}).
\btitle{Sub-fractional {B}rownian motion and its relation to
occupation times}.
\bjournal{Statist. Probab. Lett.}
\bvolume{69}
\bpages{405--419}.
\bid{doi={10.1016/j.spl.2004.06.035}, issn={0167-7152}, mr={2091760}}
\end{barticle}
%

\bptok{imsref}%
\endbibitem

\bibitem{Ch01}
%
\begin{barticle}[mr]
\bauthor{\bsnm{Cheridito},~\bfnm{Patrick}\binits{P.}}
(\byear{2001}).
\btitle{Mixed fractional {B}rownian motion}.
\bjournal{Bernoulli}
\bvolume{7}
\bpages{913--934}.
\bid{doi={10.2307/3318626}, issn={1350-7265}, mr={1873835}}
\end{barticle}
%

\bptok{imsref}%
\endbibitem

\bibitem{Ch03}
%
\begin{barticle}[mr]
\bauthor{\bsnm{Cheridito},~\bfnm{Patrick}\binits{P.}}
(\byear{2003}).
\btitle{Arbitrage in fractional {B}rownian motion models}.
\bjournal{Finance Stoch.}
\bvolume{7}
\bpages{533--553}.
\bid{doi={10.1007/s007800300101}, issn={0949-2984}, mr={2014249}}
\end{barticle}
%

\bptok{imsref}%
\endbibitem

\bibitem{Ch03b}
%
\begin{bincollection}[mr]
\bauthor{\bsnm{Cheridito},~\bfnm{Patrick}\binits{P.}}
(\byear{2003}).
\btitle{Representations of {G}aussian measures that are equivalent to
{W}iener measure}.
In \bbooktitle{S\'eminaire de {P}robabilit\'es {XXXVII}}.
\bseries{Lecture Notes in Math.}
\bvolume{1832}
\bpages{81--89}.
\bpublisher{Springer},
\blocation{Berlin}.
\bid{doi={10.1007/978-3-540-40004-2_3}, mr={2053042}}
\end{bincollection}
%

\bptok{imsref}%
\endbibitem

\bibitem{Cr64}
%
\begin{barticle}[mr]
\bauthor{\bsnm{Cram{\'e}r},~\bfnm{Harald}\binits{H.}}
(\byear{1964}).
\btitle{Stochastic processes as curves in {H}ilbert space}.
\bjournal{Teor. Verojatnost. i Primenen.}
\bvolume{9}
\bpages{193--204}.
\bid{issn={0040-361X}, mr={0170375}}
\end{barticle}
%

\bptok{imsref}%
\endbibitem

\bibitem{EDV}
%
\begin{bbook}[mr]
\bauthor{\bsnm{Edwards},~\bfnm{R.~E.}\binits{R.~E.}}
(\byear{1965}).
\btitle{Functional Analysis. {T}heory and Applications}.
\bpublisher{Holt, Rinehart and Winston},
\blocation{New York}.
\bid{mr={0221256}}
\end{bbook}
%

\bptok{imsref}%
\endbibitem

\bibitem{GK70}
%
\begin{bbook}[mr]
\bauthor{\bsnm{Gohberg},~\bfnm{I.~C.}\binits{I.~C.}} \AND
\bauthor{\bsnm{Kre{\u\i}n},~\bfnm{M.~G.}\binits{M.~G.}}
(\byear{1970}).
\btitle{Theory and Applications of {V}olterra Operators in {H}ilbert Space}.
\bpublisher{Amer. Math. Soc.},
\blocation{Providence, RI}.
\bid{mr={0264447}}
\end{bbook}
%

\bptok{imsref}%
\endbibitem

\bibitem{He11}
%
\begin{bincollection}[mr]
\bauthor{\bsnm{Heunis},~\bfnm{A.~J.}\binits{A.~J.}}
(\byear{2011}).
\btitle{The innovations problem}.
In \bbooktitle{The {O}xford Handbook of Nonlinear Filtering}
\bpages{425--449}.
\bpublisher{Oxford Univ. Press},
\blocation{Oxford}.
\bid{mr={2884604}}
\end{bincollection}
%

\bptok{imsref}%
\endbibitem

\bibitem{Hida60}
%
\begin{barticle}[mr]
\bauthor{\bsnm{Hida},~\bfnm{Takeyuki}\binits{T.}}
(\byear{1960/1961}).
\btitle{Canonical representations of {G}aussian processes and their
applications.}
\bjournal{Mem. Coll. Sci. Univ. Kyoto Ser. A. Math.}
\bvolume{33}
\bpages{109--155}.
\bid{mr={0119246}}
\end{barticle}
%

\bptok{imsref}%
\endbibitem

\bibitem{HH76}
%
\begin{bbook}[mr]
\bauthor{\bsnm{Hida},~\bfnm{Takeyuki}\binits{T.}} \AND
\bauthor{\bsnm{Hitsuda},~\bfnm{Masuyuki}\binits{M.}}
(\byear{1993}).
\btitle{Gaussian Processes}.
\bseries{Translations of Mathematical Monographs}
\bvolume{120}.
\bpublisher{Amer. Math. Soc.},
\blocation{Providence, RI}.
\bid{mr={1216518}}
\end{bbook}
%

\bptok{imsref}%
\endbibitem

\bibitem{Hi68}
%
\begin{barticle}[mr]
\bauthor{\bsnm{Hitsuda},~\bfnm{Masuyuki}\binits{M.}}
(\byear{1968}).
\btitle{Representation of {G}aussian processes equivalent to {W}iener process}.
\bjournal{Osaka J. Math.}
\bvolume{5}
\bpages{299--312}.
\bid{issn={0030-6126}, mr={0243614}}
\end{barticle}
%

\bptok{imsref}%
\endbibitem

\bibitem{HV02}
%
\begin{bincollection}[mr]
\bauthor{\bsnm{Houdr{\'e}},~\bfnm{Christian}\binits{C.}} \AND
\bauthor{\bsnm{Villa},~\bfnm{Jos{\'e}}\binits{J.}}
(\byear{2003}).
\btitle{An example of infinite dimensional quasi-helix}.
In \bbooktitle{Stochastic Models ({M}exico {C}ity, 2002)}.
\bseries{Contemp. Math.}
\bvolume{336}
\bpages{195--201}.
\bpublisher{Amer. Math. Soc.},
\blocation{Providence, RI}.
\bid{doi={10.1090/conm/336/06034}, mr={2037165}}
\end{bincollection}
%

\bptok{imsref}%
\endbibitem

\bibitem{Kai68}
%
\begin{barticle}[mr]
\bauthor{\bsnm{Kailath},~\bfnm{Thomas}\binits{T.}}
(\byear{1968}).
\btitle{An innovations approach to least-squares estimation. I.
{L}inear filtering in additive white noise}.
\bjournal{IEEE Trans. Automat. Control}
\bvolume{AC-13}
\bpages{646--655; comment, ibid. \textbf{AC-15} (1970), 158--159}.
\bid{issn={0018-9286}, mr={0309257}}
\bptnote{check pages}%
\end{barticle}
%

\bptok{imsref}%
\endbibitem

\bibitem{Kai70}
%
\begin{barticle}[mr]
\bauthor{\bsnm{Kailath},~\bfnm{Thomas}\binits{T.}}
(\byear{1970}).
\btitle{Likelihood ratios for {G}aussian processes}.
\bjournal{IEEE Trans. Inform. Theory}
\bvolume{IT-16}
\bpages{276--288}.
\bid{issn={0018-9448}, mr={0465488}}
\end{barticle}
%

\bptok{imsref}%
\endbibitem

\bibitem{KaiPoor98}
%
\begin{barticle}[mr]
\bauthor{\bsnm{Kailath},~\bfnm{Thomas}\binits{T.}} \AND
\bauthor{\bsnm{Poor},~\bfnm{H.~Vincent}\binits{H.~V.}}
(\byear{1998}).
\btitle{Detection of stochastic processes}.
\bjournal{IEEE Trans. Inform. Theory}
\bvolume{44}
\bpages{2230--2259}.
\bid{doi={10.1109/18.720538}, issn={0018-9448}, mr={1658799}}
\end{barticle}
%

\bptok{imsref}%
\endbibitem

\bibitem{KO73}
%
\begin{barticle}[mr]
\bauthor{\bsnm{Kallianpur},~\bfnm{G.}\binits{G.}} \AND
\bauthor{\bsnm{Oodaira},~\bfnm{H.}\binits{H.}}
(\byear{1973}).
\btitle{Non-anticipative representations of equivalent {G}aussian processes}.
\bjournal{Ann. Probab.}
\bvolume{1}
\bpages{104--122}.
\bid{mr={0420822}}
\end{barticle}
%

\bptok{imsref}%
\endbibitem

\bibitem{Kress}
%
\begin{bbook}[mr]
\bauthor{\bsnm{Kress},~\bfnm{Rainer}\binits{R.}}
(\byear{2014}).
\btitle{Linear Integral Equations},
\bedition{3rd} ed.
\bseries{Applied Mathematical Sciences}
\bvolume{82}.
\bpublisher{Springer},
\blocation{New York}.
\bid{doi={10.1007/978-1-4614-9593-2}, mr={3184286}}
\end{bbook}
%

\bptok{imsref}%
\endbibitem

\bibitem{LS1}
%
\begin{bbook}[mr]
\bauthor{\bsnm{Liptser},~\bfnm{Robert~S.}\binits{R.~S.}} \AND
\bauthor{\bsnm{Shiryaev},~\bfnm{Albert~N.}\binits{A.~N.}}
(\byear{2001}).
\btitle{Statistics of Random Processes. {I}},
\bedition{expanded} ed.
\bseries{Applications of Mathematics (New York)}
\bvolume{5}.
\bpublisher{Springer},
\blocation{Berlin}.
\bid{mr={1800857}}
\end{bbook}
%

\bptok{imsref}%
\endbibitem

\bibitem{LS89}
%
\begin{bbook}[mr]
\bauthor{\bsnm{Liptser},~\bfnm{R.~Sh.}\binits{R.~Sh.}} \AND
\bauthor{\bsnm{Shiryayev},~\bfnm{A.~N.}\binits{A.~N.}}
(\byear{1989}).
\btitle{Theory of Martingales}.
\bseries{Mathematics and Its Applications (Soviet Series)}
\bvolume{49}.
\bpublisher{Kluwer Academic},
\blocation{Dordrecht}.
\bid{doi={10.1007/978-94-009-2438-3}, mr={1022664}}
\end{bbook}
%

\bptok{imsref}%
\endbibitem

%

\bibitem{MR99}
%
\begin{barticle}[mr]
\bauthor{\bsnm{Marinucci},~\bfnm{D.}\binits{D.}} \AND
\bauthor{\bsnm{Robinson},~\bfnm{P.~M.}\binits{P.~M.}}
(\byear{1999}).
\btitle{Alternative forms of fractional {B}rownian motion}.
\bjournal{J. Statist. Plann. Inference}
\bvolume{80}
\bpages{111--122}.
\bid{doi={10.1016/S0378-3758(98)00245-6}, issn={0378-3758}, mr={1713794}}
\end{barticle}
%

\bptok{imsref}%
\endbibitem

\bibitem{NVV99}
%
\begin{barticle}[mr]
\bauthor{\bsnm{Norros},~\bfnm{Ilkka}\binits{I.}},
\bauthor{\bsnm{Valkeila},~\bfnm{Esko}\binits{E.}} \AND
\bauthor{\bsnm{Virtamo},~\bfnm{Jorma}\binits{J.}}
(\byear{1999}).
\btitle{An elementary approach to a {G}irsanov formula and other
analytical results on fractional {B}rownian motions}.
\bjournal{Bernoulli}
\bvolume{5}
\bpages{571--587}.
\bid{doi={10.2307/3318691}, issn={1350-7265}, mr={1704556}}
\end{barticle}
%

\bptok{imsref}%
\endbibitem

\bibitem{PT01}
%
\begin{barticle}[mr]
\bauthor{\bsnm{Pipiras},~\bfnm{Vladas}\binits{V.}} \AND
\bauthor{\bsnm{Taqqu},~\bfnm{Murad~S.}\binits{M.~S.}}
(\byear{2001}).
\btitle{Are classes of deterministic integrands for fractional
{B}rownian motion on an interval complete?}
\bjournal{Bernoulli}
\bvolume{7}
\bpages{873--897}.
\bid{doi={10.2307/3318624}, issn={1350-7265}, mr={1873833}}
\end{barticle}
%

\bptok{imsref}%
\endbibitem

\bibitem{RN55}
%
\begin{bbook}[mr]
\bauthor{\bsnm{Riesz},~\bfnm{Frigyes}\binits{F.}} \AND
\bauthor{\bsnm{Sz.-Nagy},~\bfnm{B{\'e}la}\binits{B.}}
(\byear{1990}).
\btitle{Functional Analysis}.
\bpublisher{Dover},
\blocation{New York}.
\bid{mr={1068530}}
\end{bbook}
%

\bptok{imsref}%
\endbibitem

\bibitem{Roz77}
%
\begin{bbook}[mr]
\bauthor{\bsnm{Rozanov},~\bfnm{Yuriy~A.}\binits{Y.~A.}}
(\byear{1977}).
\btitle{Innovation Processes}.
\bpublisher{V. H. Winston \& Sons},
\blocation{Washington, DC}.
\bid{mr={0445595}}
\end{bbook}
%

\bptok{imsref}%
\endbibitem

\bibitem{Sh66}
%
\begin{barticle}[mr]
\bauthor{\bsnm{Shepp},~\bfnm{L.~A.}\binits{L.~A.}}
(\byear{1966}).
\btitle{Radon--{N}ikod\'ym derivatives of {G}aussian measures}.
\bjournal{Ann. Math. Statist.}
\bvolume{37}
\bpages{321--354}.
\bid{issn={0003-4851}, mr={0190999}}
\end{barticle}
%

\bptok{imsref}%
\endbibitem

\bibitem{V93}
%
\begin{bbook}[mr]
\bauthor{\bsnm{Vainikko},~\bfnm{Gennadi}\binits{G.}}
(\byear{1993}).
\btitle{Multidimensional Weakly Singular Integral Equations}.
\bseries{Lecture Notes in Math.}
\bvolume{1549}.
\bpublisher{Springer},
\blocation{Berlin}.
\bid{mr={1239176}}
\end{bbook}
%

\bptok{imsref}%
\endbibitem

\bibitem{VP80}
%
\begin{barticle}[mr]
\bauthor{\bsnm{Vainikko},~\bfnm{G.}\binits{G.}} \AND
\bauthor{\bsnm{Pedas},~\bfnm{A.}\binits{A.}}
(\byear{1980/1981}).
\btitle{The properties of solutions of weakly singular integral equations}.
\bjournal{J. Austral. Math. Soc. Ser. B}
\bvolume{22}
\bpages{419--430}.
\bid{doi={10.1017/S0334270000002769}, issn={0034-2700}, mr={0626933}}
\end{barticle}
%

\bptok{imsref}%
\endbibitem

\bibitem{vZ07}
%
\begin{barticle}[mr]
\bauthor{\bparticle{van} \bsnm{Zanten},~\bfnm{Harry}\binits{H.}}
(\byear{2007}).
\btitle{When is a linear combination of independent f{B}m's equivalent
to a single f{B}m?}
\bjournal{Stochastic Process. Appl.}
\bvolume{117}
\bpages{57--70}.
\bid{doi={10.1016/j.spa.2006.05.013}, issn={0304-4149}, mr={2287103}}
\end{barticle}
%

\bptok{imsref}%
\endbibitem

\bibitem{vZ08}
%
\begin{barticle}[mr]
\bauthor{\bparticle{van} \bsnm{Zanten},~\bfnm{Harry}\binits{H.}}
(\byear{2008}).
\btitle{A remark on the equivalence of {G}aussian processes}.
\bjournal{Electron. Commun. Probab.}
\bvolume{13}
\bpages{54--59}.
\bid{doi={10.1214/ECP.v13-1348}, issn={1083-589X}, mr={2386062}}
\end{barticle}
%

\bptok{imsref}%
\endbibitem
\end{thebibliography}
\end{document}